\documentclass[psamsfonts]{amsart}

\usepackage{amssymb,amsfonts}
\usepackage[all,arc]{xy}
\usepackage{enumerate}
\usepackage{mathrsfs}
\usepackage{graphicx}

\newtheorem{thm}{Theorem}[section]
\newtheorem{cor}[thm]{Corollary}
\newtheorem{prop}[thm]{Proposition}
\newtheorem{lem}[thm]{Lemma}

\theoremstyle{definition}

\theoremstyle{remark}

\newlength{\figboxwidth}
\newcommand{\makefig}[3]{
        \begin{figure}[htb]
        \refstepcounter{figure}
        \label{#2}
        \begin{center}
                #3~\\
                \smallskip
                Figure \thefigure.  #1
        \end{center}
        \medskip
        \end{figure}
}

\makeatletter
\let\c@equation\c@thm
\makeatother
\numberwithin{equation}{section}

\title[Convex Cocompact Subgroups]{Dynamics of Convex Cocompact Subgroups of Mapping Class Groups}

\author{Ilya Gekhtman}

\date{6 December 2012}

\begin{document}

\begin{abstract}
For a convex cocompact subgroup $G<Mod(S)$, and points $x,y \in Teich(S)$ we obtain asymptotic formulas as $R\to \infty$ of $|B_{R}(x)\cap Gy|$ as well as the number of conjugacy classes of pseudo-Anosov elements in $G$ of dilatation at most $R$. We do this by developing an analogue of Patterson-Sullivan theory for the action of $G$ on $PMF$.

\end{abstract}

\maketitle

\tableofcontents

\section{Statement of Results}
The study of the dynamics of the action of the mapping class group on Teichm$\ddot{\mathrm{u}}$ller space has long been influenced by analogy with the actions of discrete isometry groups of manifolds of negative curvature. Two properties of interest for a negatively curved manifold $M$ are the growth of orbits of $\pi_{1}(M)$ and the asymptotics as $R\to \infty$ of the number $n_{M}(R)$ of closed geodesics of length at most $R$. For $M$ compact and negatively curved and $x,y\in M$, Margulis \cite{Marg 70} showed in his 1970 thesis that
$$\lim_{R\to\infty}e^{-hR}|B_{R}(x)\cap \Gamma y|=\Lambda(x)\Lambda(y)$$  and $$\lim_{R\to\infty}hRe^{-hR}n_{M}(R)=1$$ where $\Lambda$ is a continuous function on $M$ and $h$ is the topological entropy of the geodesic flow. Roblin \cite{Rob} generalized these results to $M=X/\Gamma$ any quotient of a contractible $CAT(-1)$ metric space by a geometrically finite group $\Gamma$. 
We prove an analogue of Roblin's result for certain subgroups of mapping class groups acting on Teichm$\ddot{\mathrm{u}}$ller space.
Let $S$ be a surface of genus $g\geq 2$. Let $Teich(S)$ be the associated Teichm$\ddot{\mathrm{u}}$ller space of isotopy classes of marked complex structures on $S$. Let $Mod(S)=Diff(S)/Diff_{0}(S)$ be the associated mapping class group, and let $d_{T}$ denote the Teichm$\ddot{\mathrm{u}}$ller metric on $Teich(S)$.
A subgroup $G<Mod(S)$ is called convex cocompact if its orbit in $Teich(S)$ is quasiconvex. Convex-cocompact subgroups of $Mod(S)$ were introduced by Farb and Mosher \cite{FM} and further developed by Kent and Leininger \cite{KL1}.
We prove:
\begin{thm}
Let S be a closed surface of genus $g\geq 2$. Let $G<Mod(S)$ be a convex cocompact subgroup containing a pseudo-Anosov element whose axis lies in the principal stratum. Let $x,y \in Teich(S)$ and $B_{R}(x)$ the ball of radius $R$ about $x$ in the Teichm$\ddot{\mathrm{u}}$ller metric. 
Let $h$ be the exponent of convergence of $G$ with respect to the Teichm$\ddot{\mathrm{u}}$ller metric.
Then $$\lim_{R\to\infty}e^{-hR}|B_{R}(x)\cap \Gamma y|=\Lambda(x)\Lambda(y)$$  where $\Lambda$ is some $G$ invariant continuous function on $Teich(S)$.
\end{thm}

\begin{thm}
Let $G$ be as in Theorem 1.1
Let $n_{M}(R)$ be the number of conjugacy classes of primitive  pseudo-Anosov mapping classes in $G$ of Teichm$\ddot{\mathrm{u}}$ller translation length at most $R$ (this translation length is the logarithm of the dilatation of the pseudo-Anosov representative).  
Then 
$$\lim_{R\to\infty}hRe^{-hR}n_{M}(R)=1.$$  
\end{thm}
Like Margulis and Roblin, we prove the counting estimate by constructing a certain measure on the unit tangent bundle and prove it is mixing. We develop an analogue of Patterson-Sullivan theory for the action of $G<Mod(S)$ on Thurston's sphere $PMF$. We construct a unique $G$-conformal density $\nu_{x},x\in Teich(S)$ supported on the limit set $\Lambda(G)\subset PMF$ of $G$, and scale the product measure $\nu_{x}\times \nu_{x}$ by a factor depending on the Busemann function to form a finite (in fact compactly supported) $G$ invariant measure $\mu$ on the unit (co)tangent bundle $Q^{1}(S)$, which can be considered the analogue Bowen-Margulis measure in negative curvature. We prove
\begin{thm}
The measure $\mu$ associated with any convex cocompact subgroup $G<Mod(S)$ is mixing.
\end{thm}

             A difficulty faced in our setting is that $Teich(S)$ is not globally hyperbolic in any reasonable sense. It is neither $CAT(0)$ nor Gromov hyperbolic: indeed pairs of geodesic rays through the same point may fellow-travel arbitrarily far apart \cite{Ivanov}.  Thurston proved that  $Teich(S)$ has a natural $Mod(S)$ equivariant compactification by the sphere of projective measured foliations, but not every geodesic ray converges to a limit in $PMF$, rays with the same limit point are not necessarily asymptotic, and rays with different limit points may stay a bounded distance apart. Thurston's compactification coincides neither with the Gromov compactification (which is not Hausdorff) nor with the horofunction compactification (which contains $PMF$ as a proper subset of smaller dimension). 
      However, the proof of Theorems 1.1 and 1.2 requires that generic (with respect to the analogue of the Bowen-Margulis measure) geodesic segments have a certain property typical of $CAT(-1)$ spaces:  namely, if two geodesics both pass within two balls of bounded radius lying far apart, they become very close somewhere in the middle. This occurs only if the segments spend a uniform proportion of time in the part of $Q^{1}(S)$  with no short flat curves (if we only required the segments to spend a uniform proportion time over a compact subset of the moduli space of Riemann surfaces, we would see behavior indicative of Gromov hyperbolicity).  We use some ergodic-theoretic arguments together with some hodge norm estimates from \cite{EM} to show that the asymptotics is controlled by geodesic segments which  are well-behaved in this sense and use techniques analogous to Roblin's to count these well-behaved geodesics.

           In order to prove mixing of $\mu$ in Theorem 1.3 we prove a certain nondegeneracy condition for the length spectrum of $G$, which is in our setting the measure of maximal entropy for the Teichm$\ddot{\mathrm{u}}$ller geodesic flow over $Teich(S)/G$.

\begin{thm}
Let $G<Mod(S)$ be a nonelementary subgroup. Then the logarithms of the dilatations of pseudo-Anosov elements of $G$ generate a dense subgroup of $\mathbb{R}$.
\end{thm}
For subsemigroups of $SL_{n}\mathbb{R}$ acting irreducibly on $\mathbb{R}^{n}$ and containing a proximal element, an analogous result is proved by Guivarch and Urban in \cite{GuUr}. In variable negative curvature this question remains open. We prove Theorem 1.4 by using the affine and symplectic structure of $MF$ given by train track coordinates to embedd a sub semigroup of $G$ into $SL_{n}\mathbb{R}$ with the image satisfying the conditions of \cite{GuUr}.

When $G$ is the full mapping class group, analogues of Theorems 1.1 and 1.2 respectively were proved by Athreya-Bufetov-Eskin-Mirzakhani in \cite{ABEM} and Eskin-Mirzakhani in \cite{EM}, in which case the Bowen-Margulis measure coincides with the Masur-Veech measure.

\subsection{Acknowledgments}
I would like to thank my advisor, Alex Eskin, for suggesting this problem to me and his constant advice, encouragement and insight. I would like to thank Howard Masur for teaching me a lot about Teichm$\ddot{\mathrm{u}}$ller theory and critical advice on many sections of this paper. I would like to thank Barbara Schapira for clearing up my misconceptions about Patterson-Sullivan theory and useful comments on an earlier draft of this paper. I would like to thank Alex Wright for useful comments on an earlier draft of this paper. I would like to thank Tam Nguyen Phan for teaching me a lot about various aspects of negative and nonpositive curvature.
I would like to thank Frederic Paulin for an encouraging conversation about extending Patterson-Sullivan theory to the Teichm$\ddot{\mathrm{u}}$ller setting when I was beginning working on this project.
I would like to thank Jayadev Athreya, Francoise Dal'bo, Benson Farb, Vaibhav Gadre, Roland Gunesch, and Joseph Maher for useful conversations. 
\section{Background on Teichm$\ddot{\mathrm{u}}$ller Theory}
Let $S$ be a closed surface of genus at least $2$. Let $Mod(S)$ be the mapping class group of $S$. Let $Teich(S)$ be the space of marked complex structures on $S$ up to isotopy. The space $Q(S)$ of quadratic differentials can be thought as a cotangent bundle of $Teich(S)$. A stratum of $Q(S)$ consists of all quadratic differentials whose zeros have the same combinatorial singularity type. The principal stratum consists of all quadratic differentials with simple zeros.  Let  $Q^{1}(S)$ be the space of area one quadratic differentials, which can be identified with the unit cotangent bundle to $Teich (S)$. Let $\pi: Q^{1}(S)\to Teich(S)$ be the projection. Let $MF$ be the space of measured foliations on $S$ and $PMF$ its projectivization.  For a quadratic differential $q$ let $q^{+},q^{-}\in MF$ denote  its vertical and horizontal measured foliations respectively and $[q^{+}],[q^{-}]\in PMF$ its projective classes.  Let $UE\subseteq PMF$ denote the projective classes of uniquely ergodic foliations.  There is a compactification due to Thurston of $Teich(S)$ by $PMF$ \cite{FLP} obtained by embedding both into $\mathbb{R}^{A}$ where $A$ is the set of isotopy classes of simple closed curves on $S$. Unless otherwise stated, the topology on $Teich(S)\cup PMF$ in this paper comes from the Thurston compactification of $Teich(S)$. Given a basepoint in Teich(S), we can also compactify $Teich(S)$ by equivalence classes of geodesic rays through the point, called the Teichm$\ddot{\mathrm{u}}$ller compactification. 
 Hubbard and Masur showed in \cite{HM} that there is a homeomorphism $Q(S)\to Teich(S)\times MF$ obtained by associating to a quadratic differential its projection in $Teich(S)$ and vertical (or horizontal) measured foliation. 
In \cite{Mas 80} Masur shows:
\begin{thm}
If $q\in Q(S)$ with $q^{+}$ uniquely ergodic then $\pi(g_{t}q)$ converges to the projective class of $q^{+}$ in $PMF$.
\end{thm}
\begin{thm}
If $q_{1},q_{2}\in Q^{1}(S)$ and $q_{1}^{+}=q_{2}^{+}$ is uniquely ergodic then $$d_{T}(\pi(g_{t}q_{1}),\pi(g_{t}q_{2}))\to 0.$$
\end{thm}
The following is part of Masur's Two Boundaries Theorem \cite{Mas 82}
\begin{thm}
The identity map on $Teich(S)$ extends to a homeomorphism between $T(S)\cup UE$ in the Teichmuller compactification and 
$T(S)\cup UE$ in the Thurston compactification.
\end{thm}
\begin{thm}
Let $x=x_{0}\in Teich(S)$.
Let $x_{n}\in Teich(S)$ be a sequence converging in the Thurston compactification to a uniquely ergodic $\eta\in PMF$. Then there exists a sequence of quadratic differentials $q_{i}\in S(x)$ and $t_{i}>0$ such that $x_{i}=\pi(g_{t_{i}}q_{i})$ and the $q_{i}$ converge to $q\in S(x)$   such that $\eta= \lim_{t\to\infty}\pi(g_{t}q)$. Therefore for any fixed $m>0$, the points $\pi(g_{m}q_{i})$ converges to $\pi(g_{m}q_{i})$.
\end{thm}
The following  result of Klarreich is Prop 5.1  in \cite{Kla 99}.
\begin{prop}
Let $F_{1}$ and $F_{2}$ be topologically inequivalent minimal foliations. Let $x_{n}$ and $y_{n}$ be sequences in $Teich(S)$ converging to $F_{1}$ and $F_{2}$ respectively.  Then the geodesic segments $[x_{n},y_{n}]$ accumulate in the Teichmuller compactification to a set $s\subseteq Teich(S)\cup PMF$ such that $s\cap Teich(S)$ is a nonempty union of geodesics whose vertical and horizontal foliations are topologically equivalent to $F_{2}$ and $F_{1}$ respectively and $s\cap PMF$ consists of foliations that are equivalent to $F_{1}$ or $F_{2}$. If $x\in Teich(S)$ is fixed, then $[x, x_{n}]$ accumulate to a set  $s\subseteq Teich(S)\cup PMF$ such that $s\cap Teich(S)$ is a union of geodesic rays based at $x$ whose endpoints are topologically equivalent to $F_{1}$ and $s\cap PMF$ consists of foliations topologically equivalent to $F_{1}$. If $F_{1}$ and $F_{2}$ are topologically equivalent then $[x_{n},y_{n}]$ converges to a subset of $PMF$ each element of which is topologically equivalent to $F_{1}$. 
\end{prop}
For uniquely ergodic foliations $F_{1}$ and $F_{2}$, the above proposition and Masur's two boundaries theorem implies the following.
\begin{cor}
$[x, x_{n}]$ converges uniformly on compact sets to $[x, [F_{1}])$ and $[x_{n}, y_{n}]$ converges to $([F_{1}], [F_{2}])$.
\end{cor}

The following corollary of Proposition 2.5 is proved in \cite{KL1}.
\begin{cor}
Let $\eta_{n},\zeta_{n}$ in $PMF$ converge to uniquely ergodic $\eta, \zeta \in PMF$. Then the accumulation points of $[\eta_{n},\zeta_{n}]$ in the Thurston topology are contained in $\left\{\eta,\zeta\right\}$. 
\end{cor}

Call a subgroup $G$ of $Mod(S)$ non-elementary if it contains a pair of noncommuting pseudo-Anosovs.
For a nonelementary subgroup $G$ of $Mod(S)$, let $\Lambda(G)$ denote its limit set in $PMF$,
the unique closed $G$ invariant subset of $PMF$ on which $G$ acts minimally. The limit set is the closure of the set of stable foliations of pseudo-Anosov mapping classes in G. It is perfect and has empty interior provided it is not equal to $PMF$ \cite{KL2}. 
If any pair of points in $\Lambda(G)$ fill $S$, we define $WH(G)$ to be the union of all Teichmuller geodesics whose vertical and horizontal measured foliations have projective classes in $\Lambda(G)$.
A subgroup $G$ of $Mod(S)$ is called convex cocompact if some $G$-orbit in $Teich(S)$ is quasiconvex.
The following properties of convex cocompact subgroups of $Mod(S)$ are proved in \cite{FM} and  \cite{KL1}.
\begin{thm}
\begin{itemize}
\item  Every $G$ orbit is quasi-convex.
\item The weak hull $WH(G)$ is defined and $G$ acts cocompactly on $WH(G)$.
\item Every limit point $\eta$ of $G$ is conical, that is for $x\in Teich(S)$  there is some $D>0$ such that the ray $[x,\eta)$ has infinite intersection with $D$ neighborhood of $Gx$ .
\item $G$ acts cocompactly on $WH(G) \cup \Lambda(G)$.
\item $WH(G)$ is contained in the $\epsilon$-thick part of $Teich(S)$ and $A$-quasiconvex for some $\epsilon>0$ and $A>0$.
\item $WH(G)\cap \Lambda(G)$ is closed in $Teich(S)\cap PMF$.
\item Every point of $\Lambda(G)$ is uniquely ergodic.
\item $G$ contains a finite index subgroup all of whose nonidentity elements are pseudo-Anosov.
\end{itemize}
\end{thm}
From now on, let $G$ denote a nonelementary convex cocompact subgroup of $Mod(S)$.
\begin{lem}
If $G$ is nonelementary convex cocompact, $\Lambda(G)$ coincides with $\overline{Gx}\cap PMF$ for any $x\in Teich(S)$.
\end{lem}
\begin{proof}
Let $\eta \in \Lambda(G)$ be an the attracting point of a pseudo-Anosov $g\in G$. Then $g^{n}x\to \eta$ so $\eta\in \overline{Gx}\cap PMF$. Since the fixed points of pseudo-Anosovs of $G$ are dense in $\Lambda(G)$, we have
$\Lambda(G)\subseteq \overline{Gx}\cap PMF$ (this holds for any nonelementary $G\subseteq Mod(S)$. 
For the other direction, it is proved by McCarthy and Papadopoulos in \cite{McP} that   $\overline{Gx}\cap PMF$ is contained in 
$$Z \Lambda(G)=\left\{\lambda\in PMF| \exists \beta \in \Lambda(G): i(\lambda,\beta)=0\right\}.$$ 
Since every point of $\Lambda(G)$ is uniquely ergodic, $Z \Lambda(G)=\Lambda(G)$.
\end{proof}
Let $Teich_{\epsilon}(S)$ denote the $\epsilon$ thick part of $Teich(S)$, the set of hyperbolic structures on $S$ where no closed curve has hyperbolic length less than $\epsilon$.
The following property of Teichm$\ddot{\mathrm{u}}$ller geodesics, indicative of hyperbolicity in the thick part, is proved by Rafi \cite{Ra}.
\begin{lem}
For each $A>0$ and $\epsilon>0$ there exists a constant $D>0$ such that for points $x, x', y, y'\in Teich_{\epsilon}(S)$ with $d_{T}(x,x')\leq A$ and $d_{T}(y,y')\leq A$ the geodesic segments $[x,y]$ and $[x',y']$ $D$-fellow travel in a parametrized fashion, and for $\eta\in PMF$ such that $[x,\eta)$ and $[x',\eta)$ are contained in $Teich_{\epsilon}(S)$, the geodesic rays $[x,\eta)$ and $[x',\eta)$ $D$-fellow travel in a parametrized fashion.
\end{lem}
For a subset $W$ of a metric space and $A>0$ let $N_{A}W$ denote the $A$ neighborhood of $W$.
\begin{cor}
Let $G\leq Mod(S)$ be convex cocompact.
For every $C>0$ there exists an $C'>0$ such that every geodesic with endpoints in $\Lambda(G)\cup N_{C}WH(G)$ is contained entirely in $N_{C'}WH(G)$.
\end{cor}
\begin{proof}
Let $A>0$ be such that $WH(G)\cup \Lambda(G)$ is $A$-quasiconvex. Let $\epsilon>0$ be such that $N_{C}WH(G)\subseteq Teich_{\epsilon}(S)$. (Such an $\epsilon$ exists since $G$ acts cocompactly on $WH(G)$ and therefore on $N_{C}WH(G)$).  Then for each $x,y\in N_{C}WH(G)$  there are $x', y'\in WH(G)$ with $d_{T}(x,x')\leq C$ and $d_{T}(y,y')\leq C$. Since $x,x',y,y'$ lie in $Teich_{\epsilon}(S)$ and $d_{T}(x,x')\leq C$ and $d_{T}(y,y')\leq C$, it follows from Rafi's theorem that the geodesic segments $[x,y]$ and $[x',y']$ $D$-fellow travel. Since $WH(G)$ is $A$-quasiconvex, $[x',y']$ is contained in $N_{A+C}WH(G)$ and so $[x,y]$ is contained in $N_{A+C+D}WH(G)$. The proof when one of $x,y$ lies in $\Lambda(G)$ is similar.
\end{proof} 
 The following is proved in \cite{MM}.
\begin{prop}
For every $\epsilon>0$ there exists an $\delta>0$ such that any triangle with vertices in $Teich(S)\cup PMF$ and sides contained in $Teich_{\epsilon}(S)$ is $\delta$ thin- each side is contained in a $\delta$ neighborhood of the other two.  
\end{prop}
From the proof of \cite{KL1}, Theorem 4.4 we also have the following.
\begin{prop}
For each $\epsilon>0$ there exists a $K>0$ with the following property. Suppose $x,y,z \in Teich(S) \cup PMF$ form a triangle with sides contained in the $\epsilon$ thick part of $Teich(S)$. Let $P\in [x,y]$ minimize the distance between $[x,y]$ and $z$. Then $[x,P]\cup [P,z]$ lies in a $K$ neighborhood of $[x,z]$.
\end{prop}
\section{Fixing the Quasiconvexity Constants}
Existence of the following constants is guaranteed by the above remarks. 
Fix $A>0$.\\
Let $A''>0$ be such that any geodesic between two points in $$N_{A}WH(\Lambda(G))\cup \Lambda(G)$$ is contained in $$N_{A''}\cup WH(\Lambda(G))\cup \Lambda(G)$$ 
Let $A'>0$ be such that any geodesic between two points in $$N_{A''}WH(\Lambda(G))\cup \Lambda(G)$$ is contained in $$N_{A'}\cup WH(\Lambda(G))\cup \Lambda(G)$$

Let $\epsilon>0$ be such that $N_{A'}WH(\Lambda(G))\subseteq Teich_{\epsilon}(S)$. Let $K>0$ be large enough so that any triangle in Teichm$\ddot{\mathrm{u}}$ller space with sides contained in $Teich_{\epsilon}(S)$ is $K$ thin, satisfies Proposition 2.13, and also large enough such that the shadow $pr_{\eta}B_{K}(x)$ contains an open set intersecting $\Lambda(G)$ for every $x\in N_{A'}WH(G)$ and $\eta \in \Lambda(G)$ (see section 7 below).

\section{Busemann Functions for the Teichmuller Metric}
If $x,y\in Teich(S)$, $\alpha\in MF$  is uniquely ergodic,  and $z_{n}\to [\alpha]$ in the Thurston compactification, then Miyachi \cite{Miy} showed
$$d_{T}(x,z_{n})-d_{T}(y,z_{n})\to \frac{1}{2}\log \frac{Ext_{\alpha}(x)}{Ext_{\alpha}(y)}.$$  
In particular, the limit $\beta_{[\alpha]}(x,y)=d_{T}(x,z_{n})-d_{T}(y,z_{n})$ exists and varies continuously with $[\alpha] \in UE$. 
This gives a continuous extension of the cocycle 

$$\beta_{z}(x,y)=d_{T}(x,z)-d_{T}(y,z)$$ to $Teich(S)\cup UE$.

For $\zeta,\eta \in MF$ uniquely ergodic and $x,w,z\in Teich(S)$ we can also define
$$\rho_{x}(z,w)=d(x,z)+d(x,w)-d(z,w)$$
$$\rho_{x}(z,[\zeta])=\rho_{x}([\zeta],z)=\lim_{w\to [\zeta]}d(x,z)+d(x,w)-d(z,w)=d(x,z)+\beta_{[\zeta]}(x,z)$$
$$\rho_{x}([\zeta],[\eta])=\lim_{(z,w)\to([\zeta],[\eta])}d(x,z)+d(x,w)-d(w,z)=\beta_{[\zeta]}(x,u)+\beta_{[\eta]}(x,u)$$ $$=\frac{1}{2}\log \frac{Ext_{x}\zeta Ext_{x}\eta}{i(\eta,\zeta)^{2}}$$ where $u\in Teich(S)$ is any point on the Teichmuller geodesic defined by $\zeta,\eta$. The function $\rho$ is continuous in $x\in Teich(S)$ and $\zeta,\eta\in UE$.  It can be considered as an analogue in our setting of the Gromov product. 

\section{Conformal Densities for $G$}
A conformal density for $G$ is a family $\left\{\nu_{x}| x\in Teich(S)\right\}$ of borel measures on $PMF$, each supported on $\Lambda(G)$ satisfying 
\\
(1) $$\gamma*\nu_{x}=\nu_{\gamma x}$$ for all $x\in Teich(S)$ and $\gamma\in  G$ and
\\
(2) For all $x,y\in Teich(S)$ $\nu_{x}$ and $\nu_{y}$ are absolutely continuous and satisfy 
$$\frac{d\nu_{x}}{d\nu_{y}}(\alpha)=exp(\delta(G)\beta_{\alpha}(x,y).$$
\\
Note, by the $G$-invariance of the Busemann cocycle if condition (2) is satisfied, it suffices to check condition (1) at a single $x$.
\begin{prop}
A conformal density $\nu_{x}$ for a nonelementary convex cocompact $G<Mod(S)$ has full support on $\Lambda(G)$ and has no atoms.
\end{prop}
\begin{proof}
Suppose $U\subseteq PMF$ is open and $U\cap \Lambda(G) \neq \emptyset$ but $\nu_{x}(U)=0$. Since the limit set is the closure of the set of stable (or unstable) laminations of pseudo-Anosov mapping classes in $G$, there is some pseudo Anosov $\gamma\in G$ with axis $l$ with repelling fixed point $l^{-}\in U$. Then for each $n>0$, $$\nu_{x}(\gamma^{n} U)=\nu_{\gamma^{-n}x}(U)=0$$ since $\nu_{x}$ and $\nu_{\gamma^{-n}x}$ are absolutely continuous. Note $$\bigcup_{n>0}\gamma^{n}U=PMF\setminus l^{+}$$ By countable subadditivity of the measure, $\nu_{x}$ is concentrated on the single point $l^{+}$. However, since $G$ is not elementary, there is some $h\in G$ with $h l^{+}\neq l^{+}$, and by absolute continuity we must also have $$\nu_{x}(h l^{+})=\nu_{h^{-1}x}(l^{+}) \neq 0$$ giving a contradiction. Thus we have proved that $\nu_{x}$ has full support on $\Lambda(G)$. Now, suppose $\nu_{x}$ has an atom $\eta\in \Lambda(G)$, say of mass r. By $[KL1]$ every limit point $\eta$ of $G$ is conical, that is there exists a $D>0$ such that the $D$ neighborhood of the geodesic $[x,\eta)$ intersects the orbit $G x$ infinitely many times. Let $\gamma_{n}\in G$ be such a sequence. Then by the triangle inequality, $\beta_{\eta}(\gamma_{n}x,x)\to \infty$.
Then, $$\nu_{x}(\gamma^{-1}_{n} \eta)=\nu_{\gamma_{n}x}(\eta)=exp(\delta(G)\beta_{\eta}(\gamma_{n}x,x))\nu_{x}(\eta)\to \infty$$ contradicting the finiteness of $\nu_{x}$.
\end{proof} 
\section{Patterson-Sullivan Construction of a Conformal density}
Let $\delta_G$ be the exponent of convergence of $G$. For $s>\delta_G$ and $x,y\in Teich(S)$ let 
$$f_s(x,y)=\sum_{\gamma\in G}\exp(-s d(x,\gamma(y))).$$ Fix $x\in Teich(S)$.
Now let 
$$\nu_{x,s}=f_{s}(x,x)^{-1}\sum_{\gamma \in G}\exp(-sd(x,\gamma(x)))\delta_{\gamma x}$$ where $\delta_{p}$ denotes the dirac measure at $p$. Now, consider a weak-* limit $\nu_{x}$ of the $\nu_{x,s}$ as $s\to \delta_G$. It is a probability measure on $$\overline{Teich(S)}=Teich(S)\cup PMF.$$ Assume first that the Poincare series diverges at $\delta_G$. Then $f(s)\rightarrow \infty$, so by discreteness of $G$, $\nu_{x}$ gives zero measure to compact subsets of $Teich(S)$, and thus must be supported on $PMF$.

Furthermore, since each $\nu_{x,s}$ is supported on the $G$ orbit of
 $x$, it follows that $\nu_{x}$ is supported on its closure, so it must be supported on $\Lambda(G)$. For any other $y\in Teich(S)$ define 

$$d\nu_{y}(\alpha)=\exp (\delta(G)\beta_{\alpha}(x,y))d\nu_{x}.$$ Since

 $\beta_{\alpha}(x,z)=\beta_{\alpha}(x,y)+\beta_{\alpha}(y,z)$

 we have that  $$\frac{d\nu_{y}}{d\nu_{z}}(\alpha)=\exp(\delta(G)\beta_{\alpha}(z,y)).$$  Now, we show that this gives a conformal density. Indeed, for any  $g\in G$ we have,

$$g\nu_{gx,s}(z)=\\exp(-s \beta_{z}(gx,x))\nu_{x,s}(z)$$ where $\beta$ is the Busemann cocycle, and taking limits as $s\to\delta$
we get  $$g d\nu_{x}(\alpha)=\exp(-\delta(G) \beta_{\alpha}(gx,x)) d\nu_{x}$$ so we indeed have a conformal density. 

Now, suppose the Poincare series converges at $\delta(G)$ (this case turns out to be vacuous, but the construction of a conformal density is required to show it). There exists a slowly growing function $h$ on $\mathbb{R}$ such that 

$$\sum_{\gamma\in G}h(d(x,\gamma x)) \\exp(-s d(x,\gamma(y)))$$ 

diverges at $s=\delta(G)$ but converges for $s<\delta(G)$.
We then set 
$$f_s(x,x)=\sum_{\gamma\in G}h(d(x,\gamma x) \\exp(-s d(x,\gamma(x)))$$
and carry out the construction as before. The existence of an appropriate function $h$ is guaranteed by application of the following result of \cite{Sul} to the Radon measure 
$$\sum_{\gamma\in \Gamma}  D_{d(x,\gamma x)}.$$
\begin{lem}
Let $\lambda$ be a Radon measure on $\mathbb{R}_{+}$, such that the Laplace transform of $\lambda$ $$\int_{\mathbb{R}_{+}} e^{-st}d\lambda(t)$$ has critical exponent $\delta\in \mathbb{R}$. Then there exists a nondecreasing function $h:\mathbb{R}_{+}\to \mathbb{R}_{+}$ such that  $$\int_{\mathbb{R}_{+}}h(t) e^{-\delta t}d\lambda(t)=\infty$$ and for every $\epsilon>0$ there exists $t_{0}\geq 0$ such that, for any $u\geq 0$ and $t\geq t_{0}$, one has $$h(u+t)\leq e^{\epsilon t}h(t)$$
In particular the Laplace transform of $h \lambda$ has critical exponent $\delta$.
\end{lem}
\section{Sectors and Closures}

For $x, y\in T(S)$ and $\eta\in PMF(S)$ let $pr_{\eta}(x)\in PMF(S)$ denote the vertical projective measured foliation of the quadratic differential $q\in S(x)$ with horizontal projective measured foliation $\eta$ and $pr_{y}(x)\in PMF(S)$ the vertical projective measured foliation of the geodesic segment from $x$ to $y$.

\begin{lem}
The function $pr_{.}(*)$ is continuous on $(Teich(S)\cup UE)\times Teich(S)$.
\end{lem}
\begin{proof}
The continuity on $Teich(S)\times Teich(S)$ and $PMF\times Teich(S)$ follows from the fact that the map $S(x)\to PMF$ is a homeomorphism for $x\in Teich(S)$.
If $x_{n}\to \eta$ uniquely ergodic, then $[x_{n},x]$ converges to $(\eta,x]$. Let $y_{n}\in [x_{n},x]$ be at distance one from $x$. Then $y_{n}\to y=\gamma_{x,\eta}(1)$ so $pr_{x_{n}}(x)=pr_{y_{n}}(x)\to pr_{y}(x)=pr_{\eta}(x)$. 
\end{proof}

For $x\in Teich(S)$ and $U\subseteq PMF$ let $Sect_{x}(U)$ be the set of all $y\in Teich(S)$ with $pr_{x}(y)\in U$
For any $r>0$ let
$$C^{+}_{r}(x,U)=N_{r}\bigcup_{z\in B_{r}(x)}Sect_{x}(U)$$ and 
$$C^{-}_{r}(x,U)=\left\{y\in Teich(S)|B(y,r)\subseteq \bigcap_{z\in B(x,r)}Sect_{z}(U)\right\}$$

\begin{lem}
For any $r>0$, $D>0$ and $U\subset PMF$ open
the closure of $C^{\pm}_{r}(x, U)\cap N_{D}WH(G)$ in the Thurston compactification is $$(C^{\pm}_{r}(x, \overline{U}^{PMF})\cap N_{D}WH(G))\cup (\overline{U}^{PMF}\cap \Lambda(G))$$
and
the closure of $Sect_{x}(U)\cap N_{D}WH(G)$ in the Thurston compactification is $$Sect_{x}(\overline{U}^{PMF})\cap N_{D}WH(G))\cup (\overline{U}^{PMF}\cap \Lambda(G))$$
 
\end{lem}
\begin{proof}
Since $pr$ is continuous, $Sect_{x}\overline{U}^{PMF}$ and  $C^{\pm}_{r}(x,\overline{U}^{PMF})$ are the closures in $Teich(S)$ of $Sect_{x}(U)$ and $C^{\pm}_{r}(x, U)$ respectively. Since $N_{D}WH(G)$ is closed in $Teich(S)$ we have that
$Sect_{x}\overline{U}^{PMF}\cap N_{D}WH(G)$ and  $C^{\pm}_{r}(x,\overline{U}^{PMF})\cap N_{D}WH(G)$ are the closures in $Teich(S)$ of $Sect_{x}(U)\cap N_{D}WH(G)$ and $C^{\pm}_{r}(x, U)\cap N_{D}WH(G)$ respectively.

Now let $x_{n}\in C^{+}_{r}(x,U)\cap N_{D}WH(G)$ converge to some $\eta\in PMF$. Since $x_{n}\in N_{D}WH(G)$ we have $\eta \in \Lambda(G)$. By definition of $C^{+}_{r}$ there exists a sequence $y_{n}\in B_{r}(x_{n})\cap Sect_{x} U$.
Since $\eta$ is uniquely ergodic we have $y_{n}\to \eta$. Now let $z_{n}\in [x,x_{n}]$ with $d(z_{n},x)=1$ and  $z\in [x,\eta)$ with $d(z_{n},x)=1$. Since $\eta$ is uniquely ergodic we have $z_{n}\to z$ and by continuity of $pr$ we have 
$pr_{x}z_{n}\to pr_{x}z=\eta$. By definition $pr_{x}z_{n}\in U$ so $\eta \in \overline{U}$. 

Now, suppose $\eta \in \Lambda(G)\cap U$. Then $\pi g_{t}q_{x,\eta}\to \eta$ as $t\to\infty$. Let $D'>0$ be such that $[x,\alpha)\subset N_{D'}WH(G)$ for all $\alpha \in \Lambda(G)$. Then $w_{t}=\pi g_{t}q_{x,\eta}\in N_{D'}WH(G)$ for all $t$. Let $v_{t}\in WH(G)\cap B_{D'}(w_{t})$. As $\eta \in UE$  we have $z_{t}\to \eta$. We claim $B_{r+D}(w_{t}\subset Sect_{x}(U)$ for large enough $t$, whence it will follow that $v_{t}\in C^{-}_{r}(x,U)$ for large enough $t$.
Indeed, otherwise, letting $V\subset PMF$  be an open neighborhood of $\eta$ with closure in $U$ there is a sequence $t_{n}\to\infty$ and $p_{n}\in B_{r+D'}(w_{t_{n}})\setminus Sect_{x}U$.  Since $\eta\in UE$ we have $p_{n}\to \eta$ and thus $pr_{x} p_{n}\to \eta$. Thus $pr_{x} p_{n}\in U$ for large enough $n$ since $U$ is open in $PMF$, contradicting our assumption. So $v_{t}\in WH(G)\cap C^{-}_{r}(x,U)$ converges to $\eta$. 

\end{proof}
\begin{cor}
If $U,V\subset PMF$ with $\overline{U}\subset V^{o}$ then $(C^{+}_{r}(x, U)\cap N_{D}WH(G))\setminus C^{-}_{r}(y, V)$ has compact closure in $Teich(S)$.
\end{cor}
\begin{cor}
If $V$ is open in $Teich(S)\cup PMF$ and $U\subset PMF$ with $\overline{U}\subset V\cap PMF$ then $C^{+}_{r}(x,U)\setminus V$ has compact closure in $Teich(S)$. If $X$ is closed in $Teich(S)\cup PMF$ and $U\subset PMF$ is open in $PMF$ with $X\cap PMF\subset U$ then $X\setminus C^{-}_{r}(x,U)$ has compact closure in $Teich(S)$.
\end{cor}

Let $V\subset PMF$ be open with $\overline{U}\subset V$.
\begin{lem}
For any geodesic $l$ with endpoints in $\Lambda(G)\setminus V$, $l\cap Sect_{x}(U)$ is contained in a compact subset of $Teich(S)$.
(In particular $l$ spends only a finite amount of time in $Sect_{x}U$).
\end{lem}
\begin{proof}

Suppose $l\cap Sect_{x}(U)$ is not contained in a compact subset of $Teich(S)$. Then (for a correct choice of orientation of the geodesic) there exist $s_{n}\to\infty$ such that $p_{n}=l_{s_{n}}\in Sect_{x}U$. Note, $l_{s_{n}}\to l^{+}\in \Lambda(G)\setminus V$, which is uniquely ergodic.
Let $q_{n}\in S(x)$ and $t_{n}>0$ be such that $p_{n}=\pi(g_{t_{n}}q_{n})$. Note, as $p_{n}\in Sect_{x}(U)$ we have $[q_{n}+]\in \overline{U}$. Then $q_{n}\to q\in S(x)$ with $$l^{+}=[q^{+}]=\lim_{t\to\infty}\pi(g_{t}q)$$ But since the map 
$$S(x)\to PMF$$ $$q\to [q^{+}]$$ is a homeomorphism, $[q^{+}]\in \overline{U}$ contradicting that $l^{+}\in \Lambda(G)\setminus V$.
\end{proof}
\begin{lem}
There exists a $T>0$ such that any geodesic $l$ with $l^{+},l^{-} \in  \Lambda(G)\setminus V$ and $d(x,l)\geq T$ is disjoint from $Sect_{x}U$.
\end{lem}
\begin{proof}
Suppose not. Then there is a sequence $l_{n}$ of geodesics with $l^{+}_{n},l^{-}_{n}\in  \Lambda(G)\setminus V$ and $d(x,l)>n$ and $p_{n}\in l_{n}\cap Sect_{x}U$. Passing to a subsequence, we have either that $(l^{+}_{n}, l^{-}_{n})$ converges to either a pair of distinct points $(\eta, \zeta)$ in $\Lambda(G)$ or a single point $\eta\in \Lambda(G)$. In the first case, we would have $l_{n}$ converge in the Hausdorff topology on $Teich(S)\cup PMF$ to the geodesic $l$ between $\eta$ and $\zeta$, so $d(l_{n},x)\to d(l,x)$ which would contradict $d(x,l_{n})\to \infty$. Thus, we have $l^{+}_{n}$ and $l^{-}_{n}$ converging to the same $\eta\in \Lambda(G)\setminus V$. Then, we have $l_{n}$ converging to $\eta$ and thus $p_{n}\to \eta$.  Let $q_{n}\in S(x)$ be such that $p_{n}=\pi(g_{t_{n}}q_{n})$. Note, as $p_{n}\in Sect_{x}(U)$ we have $[q_{n}^{+}]\in \overline{U}$. Then $q_{n}\to q\in S(x)$ with $$\eta=[q^{+}]=\lim_{t\to\infty}\pi(g_{t}q)$$ But $[q^{+}]\in \overline{U}$ contradicting that $\eta \in \Lambda(G)\setminus V$.
\end{proof}

\begin{lem}
There exists a $D>0$ such that for every $\eta, \zeta \in \Lambda(G)\setminus V$ the geodesic between $\eta, \zeta$ spends at most time $D$ in $Sect(U)$. 
\end{lem}
\begin{proof}
 Suppose the contrary. Then there exist a sequence of geodesics $l_{n}$ with both endpoints $l^{+}_{n}$,  $l^{-}_{n}$ outside of $V$ such that $l_{n}$ spends time at least $n$ in $Sect(U)$. We may pass to a subsequence such that one of the following holds: either $(l^{+}_{n},l^{-}_{n})$ converges to a pair of distinct endpoints $(\eta, \zeta)$ in $\Lambda(G)\setminus V$ or both converge to the same $\eta \in  \Lambda(G)\setminus V$. Suppose the first case. Then, the geodesics $l_{n}$ converge uniformly on compact sets to the geodesic $l$ with endpoints $\eta, \zeta$. Thus, this geodesic must spend an infinite amount of time inside $Sect(U)$, which is impossible if both of its endpoints are outside of $U$.  Now suppose both endpoints of $l_{n}$ converge to the same $\eta \in  \Lambda(G)\setminus V$. Then, $l_{n}$ converges to $\eta$. Let $p_{n}\in l_{n}\cap Sect_{x}(U)$. Let $q_{n}\in S(x)$ be such that $p_{n}=\pi(g_{t_{n}}q_{n})$. Note, as $p_{n}\in Sect_{x}(U)$ we have $[q_{n}+]\in \overline{U}$. Then $q_{n}\to q\in S(x)$ with $$\eta=[q^{+}]=\lim_{t\to\infty}\pi(g_{t}q)$$ But $[q_{n}+]\in \overline{U}$ contradicting that $\eta \in  \Lambda(G)\setminus V$.
\end{proof}

\section{The Bowen Margulis Measure}
Let $x\in Teich(S)$. Define a measure $\widetilde{\mu}$ on $\Lambda(G)\times \Lambda(G)$ by $$d\tilde{\mu}(\eta,\zeta)=\exp(\delta(G)\rho_{x}(\eta,\zeta))d\nu_{x}(\eta)d\nu_{x}(\zeta)$$ Note, every distinct pair of points in $\Lambda(G)$ give a Teichmuller  geodesic in $WH(G)$ so we can consider $\tilde{\mu}$ as a measure on $Q^{1}(S)$, invariant under the Teichm$\ddot{\mathrm{u}}$ller geodesic flow $g_{t}$ and supported on $WH(G)$. By continuity of $\rho$  on pairs of points in the limit set, $\tilde{\mu}$ is locally finite. 
\begin{lem}
The measure $\tilde{\mu}$ is $G$ invariant.
\end{lem}
\begin{proof}
This follows from the $G$ equivariance property of conformal densities.
\end{proof}
Thus $\tilde{\mu}$ descends to a measure  $\mu$ on $Q^{1}(S)/G$. Since $G$ acts cocompactly on $WH(G)$, it follows that $\mu$ is compactly supported, and therefore finite. We call $\mu$ the Bowen-Margulis measure. By projecting $\mu$ to $Teich(S)$ we obtain a measure on $Teich(S)$, supported on $WH(G)$.
For $q_{0}\in Q^{1}(S)$ let the strong stable (unstable) leaf associated to $q_0$, denoted by $W^{ss}(q_{0})$ (resp $W^{su}(q_{0})$) be the set of elements of  $Q^{1}(S)$ with the same vertical (resp horizontal) measured foliation as $q_0$. Let the weak stable (unstable) leaf associated to $q_0$, denoted by $W^{s}(q_{0})$ (resp $W^{u}(q_{0})$) be the set of elements of  $Q^{1}(S)$ with the same vertical (resp horizontal) projective measured foliation as $q_0$. 

The map sending each quadratic differential to its horizontal projective measured foliation restricts to a map
$$P^{q_{0}}_{-}: W^{ss}(q_{0})\to PMF$$ that is a homeomorphism between $W^{ss}(q_{0})$ and $PMF\setminus V(q_{0}^{+})$ where for a measured foliation $\alpha$, $V(\alpha)$ consists of the foliations $\theta$ such that $i(\alpha,\beta)+i(\theta,\beta)=0$ for some $\beta\in PMF$.   In particular, if $q_{0}^{+}$ is uniquely ergodic, $P^{q_{0}}_{-}$ is a homeomorphism between $W^{ss}(q_{0})$ and $PMF\setminus q_{0}^{+}$. We can define a locally finite measure on $W^{ss}(q_{0})$, denoted by $\tilde{\mu^{ss}_{q_0}}$ by pulling back the Patterson-Sullivan measure on $PMF$ and scaling by the Busemann function: 
$$d\tilde{\mu^{ss}_{q_0}}(v)=\exp(-\delta \beta_{[v^{-}]}(x,\pi(v)))d\nu_{x}([v^{-}])$$
Here, $x\in Teich(S)$ is any basepoint and $\mu^{ss}$ is independent of $x$.
Similarly, we can define measure on strong unstable horospheres $W^{su}(q_{0})$ by
$$d\tilde{\mu^{su}_{q_0}}(v)=\exp(-\delta \beta_{[v+]}(x,\pi(v)))d\nu_{x}([v^{+}])$$ 
We can also define measures on weak horospheres by integrating the measures on strong horospheres with respect to geodesic arclength.
$$\widetilde{\mu^{s}_{q_{0}}}=\widetilde{\mu^{ss}_{q_{0}}}dg_{t}$$
$$\widetilde{\mu^{u}_{q_{0}}}=\widetilde{\mu^{su}_{q_{0}}}dg_{t}$$

These project modulo $\Gamma$ to measures $\mu^{su}_{q_{0}}$ and $\mu^{u}_{q_{0}}$, $\mu^{s}_{q_{0}}$ and $\mu^{s}_{q_{0}}$. 

Note, whenever $q^{+}_{0}$ is uniquely ergodic, there is a map $$h^{ss}_{q_{0}}: Q^{1}(S)\setminus W^{u}(-q_{0})\to W^{ss}(q_{0})$$ with $$h^{s}_{q_{0}}(q)=W^{u}(q)\cap W^{ss}(q_{0})$$  which has one to one restrictions to any unstable horosphere. When $q^{+}$ is uniquely ergodic, $h^{ss}_{q_{0}}$ restricts to a homeomorphism between $S(x)\setminus p$ and $W^{ss}(q_{0})$ where $p^{+}=q^{+}_{0}$.
 We can also define a map $$h^{s}_{q_{0}}: Q^{1}(S)\setminus W^{u}(-q_{0})\to W^{s}(q_{0})$$
with $$h^{s}_{q_{0}}(q)=W^{su}(q)\cap W^{s}(q_{0})$$
Similarly define maps $h^{u}$ and $h^{su}$ 
\begin{lem}
For any $q_{1},q_{2}$ with $q^{+}_{i}$ uniquely ergodic, the restriction of $h^{s}_{q_{1}}$ to $W^{s}(q_{2})$ takes $\widetilde{\mu^{s}_{q_{2}}}$ to $\widetilde{\mu^{s}_{q_{1}}}$
\end{lem}
We can now define, for a fixed $q_{0}\in Q^{1}(S)$, a measure $d\tilde{\mu}_{q_{0}}$ on $Q^{1}(S)$ by 
$$d\tilde{\mu}_{q_{0}}=dg_{w}(g_{t} w) d\mu^{ss}_{v}(w)d\mu^{uu}_{q_0}(v)$$ ie by integrating over $g_t w$ with geodesic
 arclength $dg_w$, then integrating over all $w$ in $W^{ss}(v)$ with respect to$\mu^{ss}_{v}$,  finally integrating over all $v$ in $W^{su}(q_{0})$ with respect to $\mu^{su}_{q_{0}}$.
\begin{lem}
$\mu_{q_{0}}$ is independent of $q_{0}$ and coincides with the Bowen-Margulis measure $\tilde{\mu}$.
\end{lem}
\begin{proof}
This follows from the fact that $\nu_{x}$ is supported on the uniquely ergodic part of $PMF$ and that $\beta_{[v^{+}]}(x,\pi(w))=\beta_{[v^{+}]}(x,\pi(v))$ whenever $q\in W^{ss}(v)$ and $v^{+}$ uniquely ergodic.
\end{proof}

\section{Nonarithmeticity of the Length Spectrum}
In this section we prove 
\begin{thm}
Let $G<Mod(S)$ be a nonelementary subgroup. The the logarithms of the dilatations of pseudo-Anosov elements of $G$ generate a dense subgroup of $\mathbb{R}$.
\end{thm}

For a train track $\tau$, let $W_{\tau}\cong \mathbb{R}^{6g-6}$ be the vector space of weights on the branches of $\tau$ satisfying the switch condition. Let $V_{\tau}\subset W_{\tau}$ be the open cone assigning positive measure to each branch. Each element of $V_{\tau}$ corresponds to a measured foliation. Let $\phi_{\tau}: V_{\tau}\to MF$ be this correspondence. Let $U_{\tau}$ be the image of $\phi_{\tau}$ in $MF$ and let $\psi_{\tau}:U_{\tau}\to V_{\tau}$ be the inverse of $\phi_{\tau}$.

\begin{lem}
Let $\gamma \in Mod(S)$ be a pseudo-Anosov such that $\gamma^{+}$ is carried by the interior of the maximal recurrent train track $\tau$ and $\gamma^{-}$ is not carried by $\tau$. Then for large enough $n$ we have $\gamma^{n}\tau$ is carried by $\tau$ and $\gamma^{n}$ acts linearly on $W_{\tau}$ by a positive matrix whose largest eigenvalue is the dilatation $\lambda(\gamma)$ of $\gamma$.
\end{lem}
Let $G<Mod(S)$ be nonelementary. 

Let $\tau$ be a maximal recurrent train track and $\gamma_{1},\gamma_{2}\in G$  independent pseudo-Anosovs such that $\gamma^{+}_{i}$ are carried by $\tau$, and $\gamma^{-}_{i}$ are not carried by $\tau$; replacing $\gamma_{i}$ by high enough iterates we can assume that they preserve $U_{\tau}$.  Let $\Gamma \subset G$ be the semigroup freely generated by the $\gamma_{i}$.
Then each $\gamma \in \Gamma$ preserves $U_{\tau}$ and acts linearly on $W{\tau}$ by a positive matrix whose largest eigenvalue is the dilatation $\lambda(\gamma)$ of $\gamma$. Moreover, since the mapping class group preserves Thurston's symplectic form on $Mod(S)$ we have that each $\gamma \in \Gamma$ acts on $W_{\tau}$ by a symplectic matrix.

Note, if $A$ represents the action of $\gamma \in G$ then $E_{A}$ cannot have any elements of $V_{\tau}$, ie cannot have any vectors with all entries nonnegative. Indeed, if $v\in E_{A}$ then $[A^{n}v]$ does not converge to $[v_{A}]\in PW_{\tau}$. However, if $v\in V_{\tau}$ then $v=\psi_{\tau}\alpha$ for some measured foliation $\alpha$ with $[\alpha] \neq [\gamma^{-}]$. Thus, $[\gamma^{n}\alpha]\to [\gamma^{+}]$ and hence $$[A^{n}v]=[A^{n}\psi_{\tau} \alpha]=[\psi_{\tau} (\gamma^{n} \alpha)] \to [\psi_{\tau} \gamma^{+}]=v_{A}$$
In particular we obtain that if independent elements $\gamma_{1},\gamma_{2}\in \Gamma$ act on $W_{\tau}$ by matrices $A$ and $B$ respectively then  $\Gamma v_{B}\cap E_{A}=\emptyset$.

Thus, it suffices to prove the following result about linear semigroup actions on projective space. 
For a proximal element $A\in SL_{n}\mathbb{R}$ let $v_{A}$ be a dominant eigenvector with corresponding eigenvalue $\lambda(A)$ and $E_{A}$ the direct sum of complementary eigenspaces. We need 

\begin{thm}
Let $\Gamma$ be a semigroup of $SL_{n}\mathbb{R}$ every element of which is proximal. Suppose for any  $A,B \in \Gamma$ we have $\Gamma v_{B}\cap E_{A}=\emptyset$. Then the logarithms of maximal eigenvalues of matrices in $\Gamma$ generate a dense subgroup of $\mathbb{R}$. 
\end{thm}

Let the limit set of $\Gamma$, denoted by $L_{\Gamma}$ be the closure in $P\mathbb{R}^{n}$ of 
$$\{[v_{A}]: A\in \Gamma\}$$
Let $\widetilde{L_{\Gamma}}$ be the preimage of  $L_{\Gamma}$ in $\mathbb{R}^{n}$.
\begin{lem}
$L_{\Gamma}$ is $\Gamma$ invariant.
\end{lem}
\begin{proof}
Suppose $A,B\in \Gamma$. We need to show that $[Av_{B}]\in L_{\Gamma}$. 
Consider $$u=\lim_{n\to \infty}\frac{1}{||B||^{n}}B^{n}\in M_{N}(\mathbb{R})$$
Then $u$ is a  projection onto $\mathbb{R}v_{B}$ with $ker u=E_{B}$ and $u(v_{B})=v_{B}$
By assumption, $A v_{B}\notin E_{B}=ker u= ker Au$. Thus, $Au$ is a multiple of a projection onto $\mathbb{R}Av_{B}$.
Note, $$Au=\lim_{n\to \infty}\frac{1}{||B||^{n}}AB^{n}$$
so $[v_{AB^{n}}]\to [Av_{B}]$. Hence, $[Av_{B}]\in L_{\Gamma}$.
\end{proof}
\begin{lem}
Any $\Gamma$ invariant subspace $W$ is either contained in $\bigcap_{A\in \Gamma} E_{A}$ or contains $v_{A}$ for all $A\in \Gamma$.
\end{lem}
\begin{proof}
Suppose $\Gamma W=W$, and $B\in \Gamma$ with  $v_{B}\notin \Gamma$. Then for any $v\notin E_{B}$ we have $\lim_{k\to\infty}[B^{k}v]=[v_{B}]\notin [W]$. However, for any $v\in W$ we have $B^{k}v\in W$ and as $[W]$ is closed, any limit point of $\{[B^{k}v]:k\in \mathbb{N}\}$ is in $[W]$. Thus, if $v_{B}\notin W$ then $W\subset E_{B}$. In particular, since for any $A \in \Gamma$ we have $v_{A}\notin E_{B}$ we have that $W$ does not contain $v_{A}$ for any $A\in \Gamma$ and is thus contained in $E_{A}$ for all $A\in \Gamma$.
\end{proof}
Let $W_{\Gamma}=\oplus_{A\in \Gamma}\mathbb{R}v_{A}$ be the smallest subspace of $\mathbb{R}^{n}$ containing $\widetilde{L_{\Gamma}}$.
Since $\Gamma$ preserves $L_{\Gamma}$, it preserves $W_{\Gamma}$.
 Let $U_{\Gamma}$ be a maximal proper $\Gamma$ invariant subspace of $W_{\Gamma}$.
\begin{lem}
 $\Gamma$ acts irreducibly on $V_{\Gamma}=W_{\Gamma}/U_{\Gamma}$, and  each $A\in \Gamma$ has the same largest eigenvalue in this action as in the action on $\mathbb{R}^{n}$.
\end{lem}
\begin{proof}
The irreducibility follows from maximality of $U_{\Gamma}$. Note, by Lemma 10.4 we have $U_{\Gamma}\subset \bigcap_{A\in \Gamma} E_{A}$ does not contain any $v_{A}$ for $A\in \Gamma$ so $v_{A}+U_{\Gamma}$ is a dominant eigenvector for $A$ with eigenvalue $\lambda(A)$.
\end{proof}

\begin{lem}
For any independent $A,B\in \Gamma$ there is an integer $M>0$ such that $\Gamma_{M}= sg(A^{M},B^{M})$ acts strongly irreducibly on $V_{\Gamma_{M}}=W_{\Gamma_{M}}/U_{\Gamma_{M}}$.
\end{lem}

\begin{proof}
Now, consider independent $A,B\in \Gamma$. Since there is no infinite nested sequence of finite dimensional subspaces, there exists an $N>0$ and subspaces $U\subset W$ such that $W_{\Gamma_{M}}=W$, $U_{\Gamma_{M}}=U$ and $V_{\Gamma_{M}}=W/U=V$ for all $M\geq N$.
We know that $\Gamma_{N}=sg(A^{N},B^{N})$ acts irreducibly on $V$ and consequently so does the free group $G_{n}=<A^{N},B^{N}>$ generated by $A^{N},B^{N}$. Suppose the $\Gamma_{N}$ action on $V$ is not strongly irreducible. Let $V_{1},...,V_{n}\subset V$ be a minimal collection of subspaces of $V$ such that their union is preserved by $\Gamma_{N}$. Note, $A^{N},B^{N}$ permute $V_{1},...,V_{N}$ so there exists an $K>0$ such that $A^{MK}V_{i}=V_{i}$ and $B^{MK}V_{i}=V_{i}$  for each $i$. In particular, $\Gamma_{MK}=sg(A^{MK},B^{MK})$ preserves the proper nontrivial subspace $V_{1}\subset V=W$, contradicting the fact that $\Gamma_{MK}$ must act on $V$ irreducibly.
\end{proof}
The following is proved in \cite{GuUr}, Proposition 4.9.
\begin{lem}
Let $\Gamma<GL_{m}\mathbb{R}$ be a semigroup acting strongly irreducibly on $\mathbb{R}^{m}$ and containing a proximal element. Then the logarithms of maximal eigenvalues of proximal elements of $\Gamma$ generate a dense subgroup of $\mathbb{R}$.
\end{lem}
Theorem 9.1 follows from Lemmas 9.7 and 9.8.

\section{Ergodicity and Mixing of the Bowen Margulis Measure}
\begin{thm}
The geodesic flow $g_{t}$ on $Q^{1}(S)/G$ is ergodic with respect to the Bowen-Margulis measure $\mu$.
\end{thm}
\begin{proof}
For $f\in C_{c}(Q^{1}(S)/G)$ continuous with compact support, consider the forward and backward Birkhoff averages:
$$f^{+} (q)=\lim  \sup_{T\to\infty}\frac{1}{T}\int^{T}_{0}g_{t}qdt$$ and $$f^{-} (q)=\lim  \sup_{T\to\infty}\frac{1}{T}\int^{T}_{0}g_{-t}qdt$$ By the Birkhoff ergodic theorem these are finite and equal for almost every $q\in Q^{1}(S)/G$. Moreover, it is clear that $f^{+}$ and $f^{-}$ are invariant under  geodesic flow. Furthermore, $f^{+}$ is invariant along $W^{ss}(q)$ whenever $q^{+}$ is uniquely ergodic, and $f^{-}$ is invariant along $W^{su}(q)$ whenever $q^{-}$ is uniquely ergodic.
Suppose the measure is not ergodic. Then there exists some $f\in C_{c}(Q^{1}(S)/G)$ such that $f^{+}$ is NOT almost everywhere constant. Let $C_{1},C_{2}$ be disjoint sets whose union is $\mathbb{R}$ such that $D_{i}=(f^{+})^{-1} C_{i}$ has positive measure. Note, by Fubini's theorem and the product structure of the measure $\mu$ there exists $q_{0}$ with $q_{0}-\in \Lambda(G)$  and a set $A\subseteq W^{su}(q_{0})$ of full $\mu^{su}_{q_{0}}$ measure such that $f^{+}(v)=f^{-}(v)$ and $v^{+}\in \Lambda(v)$ for all $v\in A$. Furthermore there are sets $A_{i}(q_{0})\subseteq A$  of positive $\mu^{su}_{q_{0}}$ measure such that for all $v\in  A_{i}(q_{0})$ $W^{s}(v)$ intersects $D_{i}$ in a set of positive $\mu^{ss}_{v}\times dt$ measure. Note, since the $D_{i}$ are $g_{t}$ invariant and $f^{+}$ is constant along $W^{ss}(v)$, it follows that $W^{s}(v)\subseteq D_{i}$ for $v\in A_{i}(q_{0})$. In particular, $A_{i}(q_{0})\subseteq D_{i}$. However, as $f^{+}=f^{-}$ on $A$ this implies that $f^{-}(A_{i})(q_{0})\subseteq C_{i}$. However, as $q_{0}-$ is uniquely ergodic, $f^{-}$ is constant on $W^{su}(q_{0})$ contradicting that the $C_{i}$ are disjoint.
\end{proof}
It follows that $G$  acts ergodically on $PMF$ with $\nu_{x}\times \nu_{x}$. 
\begin{thm} 
The geodesic flow $g_{t}$ on $Q^{1}(S)/G$ is mixing with respect to the Bowen-Margulis measure $\mu$. 
\end{thm}
 Our argument  is modelled on Babbillot's argument in \cite{Bab} where an analogous result was proved for general quasi-product measures on manifolds of pinched negative curvature. 
The following result from unitary representation theory is proved in \cite{Bab}.
\begin{thm}
Let $(X, B, m, (T_{t})_{t\in A})$ be a measure preserving dynamical system where $(X,B)$ is a Borel space, $m$ a Borel measure on $X$ and $T$ an action an action of a locally compact second countable abelian group on $X$ by $m$ preserving transformations. Let $f\in L^{2}(X,m)$, and if $m$ is finite assume also $\int_{X} f dm=0$. Then, if $f\circ T_{a}$ does not converge weakly to $0$ as $a\to \infty$ in $A$, there exist a sequence $s_{n}$going to infinity in $A$ and a non-constant function $\psi\in L^{2}(X,m) $ such that $f\circ T_{s_{n}}$ and $f\circ T_{-s_{n}}$ both converge to $\psi$. 
\end{thm}
 
Suppose $\mu$ is not mixing. Then there is a continuous $G$ invariant function  $f$ with $supp(f)/G$ compact such that $$\int_{Q^{1}} f d\tilde{\mu}=0$$ and  $f\circ g_{t}$ does not converge weakly to zero. Let $s_{n}\to \infty$ and nonconstant $\psi$ be such that $\int_{Q^{1}} f dm=0$ and both $f\circ g_{s_{n}}$ and $f\circ g_{-s_{n}}$ converge weakly to $\psi$. By the Banach-Saks theorem, there exists a subsequence $t_{n}$ of $s_{n}$ such that the Cesaro averages $$A_{N}=\frac{1}{N}\sum^{N}_{n=1}f\circ g_{t_{n}}$$ and $$A_{-N}=\frac{1}{N}\sum^{N}_{n=1}f\circ g_{-t_{n}}$$ converge almost surely to $\psi$.  We first smooth out $\psi$ by considering the function $v\to \int^{c}_{0}\psi(g_{s} v)ds$. Choosing small enough $c$ guarantees  that this function remains non-constant, and it is moreover the limit of the corresponding Cesaro averages of the smoothing of $f$. By abuse of notation, we continue to call the new functions $f$ and $\psi$. Now, there exists a set $E_{0}$ of full $\mu$ measure in $\Lambda(G)\times \Lambda(G)$ such that for each $v$ on a geodesic with endpoints in $E_{0}$, the function $t\to \psi (g_{t}(v))$ is well defined and continuous. Consider the closed (a priori possibly trivial) subgroup $\mathbb{R}(q)$ of $\mathbb{R}$ given by the periods of $t\mapsto \psi(g_{t} q)$. It is clearly flow invariant, and thus gives a measurable map from $E_{0}$ into the set of closed subgroups of $\mathbb{R}$. By ergodicity of $\nu_{x}\times \nu_{x}$ it must be constant almost everywhere on $E_{0}$. Suppose this subgroup is $\mathbb{R}$. Then $\psi$ would be $g_{t}$ invariant, and thus pass to a flow invariant function on $Q^{1}/G$, which is not almost-everywhere constant. However, this contradicts the ergodicity of $\mu$. Thus the subgroup in question must be cyclic. Say it equals $k\mathbb{Z}$ on a full measure set $E_{1}\subseteq E_{0}$. Let $$\psi^{+}=\lim \sup A_{N}$$ and  $$\psi^{-}=\lim \sup A_{-N}$$ By Fubini's theorem, there is a set $E_{2}\subseteq E_{1}$ be of full measure and such that $\psi^{+}=\psi^{-}=\psi$  everywhere along every geodesic in $E_{2}$.  

Now, let $E^{-}$ be the set of $\lambda\in \Lambda(G)$ such that for $\nu_{x}$ almost everywhere
 $\alpha$, $(\lambda, \alpha)\in E_{2}$ and similarly let $E^{+}$ be the set of $\lambda\in \Lambda(G)$ such that for $\nu$ almost everywhere $\alpha$, $(\alpha,\lambda)\in E_{2}$. Again, by Fubini's theorem, $E=E_{2}\cap(E^{+}\times E^{-})$ has full measure. 

Now, let $\eta_{1}, \eta_{2}, \zeta_{1}, \zeta_{2}\in \Lambda(G)$. Choose $p_{0}\in (\eta_{1},\zeta_{1})$, $p_{1}\in (\zeta_{1},\eta_{2})$, $p_{2}\in (\eta_{2},\zeta_{2})$, $p_{3}\in (\zeta_{2},\eta_{1})$ and $p_{4}\in (\eta_{1},\zeta_{1})$ such that $p_{i}$ and $p_{i+1}$ are on the same horosphere. We claim that the distance $\tau(\eta_{1},\eta_{2},\zeta_{1},\zeta_{2})$ between $p_{0}$ and $p_{4}$ depends only on the $\eta_{i}$ and $\zeta_{i}$ and is thus independent of the position of $p_{0}$ on its geodesic. It will follow that this distance is a period of $t\mapsto f(g_{t}q_{0})$ where $q_{0}\in S(p_{0})$ with $[q_{0}^{+}]=\zeta_{1}$ and thus is contained in $k\mathbb{Z}$ for $\nu_{x}$ almost every  $\eta_{1}, \eta_{2}, \zeta_{1}, \zeta_{2}\in \Lambda(G)$.
Indeed, let $H_{\eta_{i}}$ and $H_{\zeta_{i}}$ be horospheres centered at $\eta_{i}$ and $\zeta_{i}$ respectively for $i=1,2$. Let $D_{ij}$ be the signed distance between the intersections of $(\eta_{i},\zeta_{j})$ with $H_{\eta_{i}}$ and $H_{\zeta_{j}}$, with the sign convention chosen in such a way that $D_{ij}$ is positive if $H_{\eta_{i}}$ and $H_{\zeta_{j}}$ are disjoint. Then since the geodesic flow takes horospheres to horospheres, the quantity $D_{1,1}+D_{2,2}-D_{1,2}-D_{2,1}$ is independent of the specific choice of horospheres. Moreover if the horospheres are chosen in such a way that $H_{\zeta_{1}}$ contains $p_{0}$, $H_{\eta_{2}}$ passes through $p_{1}= H_{\zeta_{1}}\cap (\zeta_{1},\eta_{2})$, $H_{\zeta_{2}}$ passes through $p_{2}= H_{\eta_{2}}\cap (\eta_{2},\zeta_{2})$, and $H(\eta_{1})$ passes through $p_{3}=H_{\zeta_{2}}\cap (\eta_{1},\zeta_{2})$, then $D_{1,1}+D_{2,2}-D_{1,2}-D_{2,1}$ reduces to the signed distance  $D_{1,1}$ between $p_{0}$ and $p_{4}=H_{\eta_{1}}\cap (\eta_{1},\zeta_{1})$. Thus, $\tau(\eta_{1},\eta_{2},\zeta_{1},\zeta_{2})=D_{1,1}+D_{2,2}-D_{1,2}-D_{2,1}$ is well-defined and continuous on quadruples of points in $\Lambda(G)$. We call it the cross ratio of the four points in $PMF$, or the cross ratio of the geodesics $(\eta_{1},\zeta_{1})$ and $(\eta_{2},\zeta_{2})$.

\begin{prop}
$$\tau(\eta_{1},\eta_{2},\zeta_{1},\zeta_{2})=\lim_{n\to\infty} d(x^{n}_{1},y^{n}_{1})+d(x^{n}_{2},y^{n}_{2})-d(x^{n}_{1},y^{n}_{2})-d(y^{n}_{1},x^{n}_{2})$$ where $x^{n}_{i},y^{n}_{i}\in Teich(S)$ with $x^{n}_{i}\to \eta_{i}$, $y^{n}_{i}\to \zeta_{i}$.
\end{prop}
\begin{proof}
Let $H_{\eta_{i}}$ and $H_{\zeta_{i}}$ be pairwise disjoint horospheres through $\eta_{i}$ and $\zeta_{i}$ respectively, and let $M_{0}$ be the complement in $Teich(S)$ of the corresponding horoballs. The intersections of the geodesics $$(\eta_{1},\zeta_{1}),(\zeta_{1},\eta_{2}), (\eta_{2},y_{2}), (y_{2},\eta_{1})$$ with $M_{0}$ consist of disjoint segments $I_{j}$ of length $d_{j},j=1,2,3,4$ respectively. Note the number $$\tau'=d_{1}+d_{2}-d_{3}-d_{4}$$ does not depend on the specific choice of horosphere and by continuity of the Busemann function on $\Lambda(G)$ depends continuously on the $\eta_{i},\zeta_{i}\in \Lambda(G)$. We claim that it is equal to $\tau(\eta_{1},\eta_{2},\zeta_{1},\zeta_{2})$. Indeed, suppose $x^{n}_{1},x^{n}_{2}, y^{n}_{1}, y^{n}_{2}$  are points in $Teich(S)$ converging to $\eta_{i}$ and $\zeta_{i}$ respectively. The segments $[x^{n}_{1},y^{n}_{1}]\to (\eta_{1},\zeta_{1})$ $[x^{n}_{1},y^{n}_{2}]\to  (\eta_{1},\zeta_{2})$ $[x^{n}_{2},y^{n}_{1}] \to (\eta_{2},\zeta_{1})$  $[x^{n}_{2},y^{n}_{2}]\to (\eta_{2},\zeta_{1})$ uniformly on compact sets. In particular, their intersection with $M_{0}$ contains four segments $I^{n}_{j}$ which converge toward the $I_{j}$.  Thus, to prove that $\tau=\tau'$ it suffices to show that the contribution to $$d(x^{n}_{1},y^{n}_{1})+d(x^{n}_{2},y^{n}_{2})-d(x^{n}_{1},y^{n}_{2})-d(y^{n}_{1},x^{n}_{2})$$ of the parts of $[x^{n}_{i},y^{n}_{j}]$ which are contained in the complement of $M_{0}$ goes to zero. By symmetry it suffices to show that if $p^{n}_{1}$ is the intersection of $[x^{n}_{1},y^{n}_{1}]$ with $H_{\eta_{1}}$,
$p^{n}_{2}$ is the intersection of $[x^{n}_{1},y^{n}_{2}]$ with $H_{\eta_{1}}$, then $$d(p^{n}_{1},x^{n}_{1})-d(p^{n}_{2},x^{n}_{1})\to 0$$ Note, as $n\to \infty$ we have $$p^{n}_{1}\to p_{1}=H_{\eta_{1}}\cap (\eta_{1},\zeta_{1})$$ and  $$p^{n}_{2}\to p_{2}=H_{\eta_{1}}\cap (\eta_{1},\zeta_{2})$$ and $\beta_{x^{n}_{1}}$ converges to $\beta_{\eta_{1}}$ uniformly on compact sets. Thus  $$\beta_{x^{n}_{1}}(p^{n}_{1},p^{n}_{2})=d(p^{n}_{1},x^{n}_{1})-d(p^{n}_{2},x^{n}_{1})\to \beta_{\eta_{1}}(p_{1},p_{2})$$ which is zero since $p_{1},p_{2}$ lie on the same horosphere based at $\eta_{1}$
\end{proof}

From the expressions of the Busemann functions in terms of extremal length, we in fact find $$\tau([\alpha_{1}],[\alpha_{2}],[\beta_{1}],[\beta_{2}])=\frac{1}{2}\log \frac{i(\alpha_{1},\beta_{1})i(\alpha_{2},\beta_{2})}{i(\alpha_{1},\beta_{2})i(\alpha_{2},\beta_{1})}$$ for any $\alpha_{i},\beta_{i}$ uniquely ergodic.
Note, $\tau$ defines a continuous function on quadruples of points in $Teich(S)\cup\Lambda(G)$
From this formula we obtain
\begin{cor}
For any pseudo-Anosov $g \in Mod(S)$, with fixed points $\eta_{1},\eta_{2}\in \Lambda(G)$ the translation distance of $\lambda$ is twice $\tau(\eta_{1},\eta_{2},\beta, g\beta)$ where $\beta$ is any uniquely ergodic point in $PMF$ distinct from the $\eta_{i}$. 
\end{cor}

As noted above, $\tau(\eta_{1},\eta_{2},\zeta_{1},\zeta_{2})$ is a period of $t\mapsto f(g_{t}q_{0})$ where $q_{0}\in S(p_{0})$ with $[q_{0}^{+}]=\zeta_{1}$ and so $\tau(\eta_{1},\eta_{2},\zeta_{1},\zeta_{2})\in k\mathbb{Z}$ for $\nu_{x}^{4}$ almost every  $\eta_{1}, \eta_{2}, \zeta_{1}, \zeta_{2}\in \Lambda(G)$. By continuity of the cross ratio and the fact that $\nu$ has full support on $\Lambda(G)$, it follows that $\tau(\eta_{1},\eta_{2},\zeta_{1},\zeta_{2})\in k\mathbb{Z}$ for $\nu_{x}^{4}$ for every  $\eta_{1}, \eta_{2}, \zeta_{1}, \zeta_{2}\in \Lambda(G)$. But this implies that the translation length of every element of $G$ is in $k\mathbb{Z}$ contradicting Theorem 9.1.

\section{Controlling the Multiple Zero Locus}
By ergodicity, $\mu$ gives full mass to a single stratum.
In the remainder of the paper, we will assume that this is the principal stratum.
Here is a sufficient condition.
\begin{prop}
If $G$ contains a pseudo-Anosov element with axis lying in the principal stratum, then $\mu$ gives full weight to the principal stratum.
\end{prop}

\begin{proof}
By proposition 5.1, $\nu_{x}$ has full support on $\Lambda(G)$ and thus $\mu$ has full support on $Q^{1}WH(G)$. Thus, any open set $U\subset Q^{1}(S)$ intersecting any geodesic with endpoints in $\Lambda(G)\times \Lambda(G)$ has positive $\mu$ measure. 
If $\gamma \in G$ has axis $g_{\gamma}$ in the principal stratum then a point $p\in g_{\gamma}$ has a neighborhood  $ U\subset Q^{1}(S)$ that is also contained in the principal stratum. Thus, the principal stratum has positive $\mu$ measure and by ergodicity of $\mu$ on $Q^{1}(S)/G$ it has full measure.
\end{proof}
 
 In this section, we show that the contribution to orbit growth of the multiple zero locus and thin parts of the principal stratum is asymptotically negligible.  For a subset $P\subset Q^{1}(S)$, $c\in (0,1)$ and $x\in Teich(S)$ let  $B_{R}(x,P,c)$ the set of points $y\in Teich(S)$ with $d_{T}(x,y)\leq R$ and the segment $[x,y]$ spending a proportion at most $c$ of the time in $P$.

 For $x,y\in Teich(S)$ let $N_{G}(x,y,P,R,c)$ denote the number of $\gamma \in G$ such that $d(x,\gamma y)\leq R$ and $[x,\gamma y]$ spending a proportion at most $c$ of the time in $P$.
Specifically, we prove:
\begin{thm}

For each $x,y\in Teich(S)$ and $\epsilon>0$ there exists a closed subset $P'\subset Q^{1}(S)$ disjoint from the multiple zero locus 
such that $$\lim \sup_{R\to \infty} N_{G}(x,y,P', R, 1/3)/e^{\delta R}\leq \epsilon $$
\end{thm}

In order to prove this we will show:

\begin{thm}
For each $\epsilon>0$ there exists a closed subset $P\subset Q^{1}(S)$ disjoint from the multiple zero locus
such that $$\lim \sup_{R\to \infty} e^{-\delta R} m(B_{R}(x,P,1/2))\leq \epsilon $$
\end{thm}

We first conclude Theorem 11.2 from Theorem 11.3.
 We will a lemma of Eskin and Mirzakhani from \cite{EM}
\begin{lem}[\cite{EM}, Lemma 5.4] Suppose $K\subset M_{g}$ is compact. Given $s>0$, there exists  constants
$L_{0}$ depending on $s$ and $K$, and $c_{0}$ depending only on $K$ with the following property. If 
$\gamma:[0,L]\to Q^{1}(S)$ is a geodesic segment (parametrized by arclength)  with endpoints above $K$ and $L>L_{0}$, $\widehat{\gamma}:[0,L']\to Q^{1}(S)$ is the geodesic segment connecting $p_{1},p_{2}\in Teich(S)$ such that $d_{T}(p_{1},\pi(\gamma(0)))<c_{0}$, $d_{T}(p_{2},\pi(\gamma(L)))<c_{0}$, and
$$|\left\{s\in [0,L]|l_{min}(\gamma(t))\geq s\right\}|>\frac{ L}{2}$$ then 
$$|\left\{s\in [0,L']|l_{min}(\widehat{\gamma}(t))\geq s/4\right\}|>\frac{ L'}{3}$$
\end{lem}
From this, we obtain:
\begin{lem}
Let $K\subset Teich(S)$ be a compact subset, and $P$ a closed subset of the principal stratum. Then there exists a closed subset $P'$ of the principal stratum containing $P$ in its interior and an $R_{0}>0$ such that for $y_{1},y_{2}\in K$ and $x\in Mod(S)K$ with $d_{T}(x,y_{i})>R_{0}$, if $[x,y_{1}]$ spends a proportion at most $1/3$ in $P'$ then $[x,y_{2}]$ spends a proportion at most $1/2$ in $P$.
\end{lem}
In particular we have:

\begin{lem}
For any $K \subset Teich(S)$ compact and $P$ a closed subset of the principal stratum, there exists a closed subset $P'$ of the principal stratum containing $P$ in its interior and an $R_{0}>0$ such that for any $x \in GK$, $y_{1},y_{2}\in K$ and $R\geq R_{0}$ we have $$N_{G}(x,y_{1},P',R, 1/3)\leq N_{G}(x,y_{2},P,R+diam(K)).$$
\end{lem}

\begin{proof}[Proof of Theorem 11.2 assuming Theorem 11.3]
Let $x\in Teich(S)$ and $P\subset Q^{1}(S)$ be  a subset of the principal stratum be such that the conclusion of Theorem 11.3 holds with $\epsilon/e^{\delta  diam(K)}$ in place of $\epsilon$, ie
$$m(B_{R}(x,P,1/2))\leq \frac{\epsilon}{e^{\delta diam(K)}} e^{\delta R}$$ for all large enough $R$.

Let $K\subset Teich(S)$ be compact and contain both $x$ and fundamental domain for the action of $G$ on $WH(G)$.

\begin{lem}
$$m(B_{R}(x,P, 1/2))=\int_{y\in K}N_{G}(x,y,P,R,1/2).$$
\end{lem}
\begin{proof}
Note, $$m(B_{R}(x,P, 1/2))=\sum_{g\in G} \int_{y\in g K} \chi_{B_{R}(x,P,1/2)}(y) dm(y)$$ $$=\int_{y\in K} \sum_{g\in G} \chi_{B_{R}(x,P,1/2)}(gy)dm(y)$$ $$=\int_{y\in K}N_{G}(x,y,P,R,1/2).$$ 
\end{proof}
Note, by Lemma 11.6 there exists a  closed subset $P'$ of the principal stratum containing $P$ in its interior and a $R_{0}>0$ such that if $R>R_{0}$, then
$$N_{G}(x,y_{1},P',R, 1/3)\leq N_{G}(x,y_{2},P,R+diam(K))$$ for any $y_{1},y_{2}\in K$.
Moreover, by Lemma 11.7 with with $R+diam(K)$ in place of $R$, for each large enough $R$ there exists a $y_{2}\in K $ such that 
$$N_{G}(x,y_{2},P,R+diam(K)), 1/2)\leq m(B_{R+diam(K)}(x,P, 1/2))$$ $$\leq \frac{\epsilon}{e^{\delta diam(K)}} e^{\delta (R+diam(K))}=\epsilon e^{\delta R}$$ completing the proof.
\end{proof}
We now consider the following measure on $MF$:
Note the space of strong stable (or unstable) horospheres based at uniquely egodic points can be identified with uniquely ergodic points of $MF$. Indeed, let $o\in Teich(S)$ be a basepoint.
If $\eta \in MF$ with $Ext_{o}\eta =1$ then $t\eta$ is identified with the horosphere $H(t\eta)=H(t,[\eta])$ based at $\eta$  such that $\beta_{[\eta]}(o,z)=t$ for each $z=z_{t,[\eta]}\in H$ (ie $Ext_{z}\eta=e^{2t}$). 
For $A\subset MF$ so that $[A]\subset PMF$ let 
$$\lambda(A)=\int_{[\eta]\in [A]}\int_{t: H(t,[\eta]) \in A}e^{\delta t}d\nu_{o}([\eta])=\int_{[\eta]\in [A], Ext_{o}\eta =1} \int_{t\eta \in A}e^{\delta t}dt d\nu_{o}([\eta])$$

\begin{lem}
The measure $\lambda$ does not depend on choice of basepoint $o\in Teich(S)$ and is $G$ invariant. It has support precisely on foliations projecting to points on $\Lambda(G)$. Moreover for all $U\subset Q^{1}(S)$

$$ \lambda(\eta ^{+} (g_{t}U))= e^{\delta t}\lambda(\eta^{+} U)$$ and $$\mu(U)=\int_{\eta \in MF}\mu^{ss}(A\cap H(\eta))d\lambda(\eta)$$
\end{lem}
\begin{proof}
The independence of basepoint and $G$ invariance follows because the $\nu_{o}$ form a conformal density, the other properties are immediate from the definition.
\end{proof}
Denote $$\overline{\lambda}(U)=\lambda(Cone (U))$$ where $Cone (U)$ is the union of segments from the origin in $MF$ to points of $U$.
For $W\subset Q^{1}(S)$ and $s>0$ let $W(s)$ denote the set of $q\in Q^{1}(S)$ such that there
exists $q'\subset W$ on the same leaf of $W^{su}$ as $q$ such that $d_{H}(q, q')<s$. For a subset $A\subset Teich(S)$ let $A(r)= Nbhd_{r}(A)$ denote the $r$ neighborhood of $A$ in the Teichm$\ddot{\mathrm{u}}$ller metric.
\begin{lem}
Let $K\subset Teich(S)$ be a fundamental domain for the action of $G$ on $WH(G)$.
Let $h>0$. Then there is a $C(h)>0$ depending only on $K$ and $h$ such that for all $U\subset Q^{1}WH(G)\cap \pi^{-1}K$ and all $t>0$ letting
$W_{t}=g_{t}U$ we have 
$$m(Nbhd_{2}\pi(W_{t}))\leq C(h)\overline{\lambda} (\eta^{+}W_{t}(h)).$$
\end{lem}
\begin{proof}
 Let $h_{0}=h_{0}(K, h)$ be a small constant to be specified later. We can decompose $U$ into pieces $U_{\alpha}$ such that each piece is within Hodge distance $h_{0}/2$ of a single unstable leaf. The minimal number of such pieces can be bounded by a constant depending only on $K$ by the compactness of $K$ and equivalence of the Euclidean and Hodge metrics over compact sets so, we may assume without loss of generality that $U$ is within Hodge distance $h_{0}/2$ of a single unstable leaf. Also, as in [ABEM, Lemma 4.1], we can assume without loss of generality that $U$ has $W^{su}\times W^{s}$ product structure.
 Pick a maximal $\Delta \subset \pi(W_{t})$ with $d_{T}(x,y)=1$ for any distinct $x,y \in \Delta$. Note by compactness of $K$ and $G$ equiariance of $m$ there is a constant $C(K)$ depending only on $K$ such that 
 $$m(B(X,3))\leq C(K)$$ for all $X\in G K$
 Then $$Nbhd_{2}\pi(W_{t})\subset \bigcup_{X\in \Delta} B_{T}(X,3),$$
 and hence 
 $$m(Nbhd_{2}\pi(W_{t}))\leq \sum_{X\in \Delta} (B_{T}(X,3))\leq |\Delta| C(K)$$
 Now, let $\Delta'\subset W_{t}$ be a set containing one element of $\pi^{-1}(X)\cap Q^{1}WH(G)$ for each $X \in \Delta$. Let $B^{su}_{E}(q,r)$ denote the elements of $W^{su}(q)$ within euclidean distance $r$ of $q$.
As shown in \cite{ABEM}, Lemma 4.1 for $h_{0}$ small enough we can pick $h_{2}$ depending only on $K$ such that the $\eta^{+}(B^{su}(q,h_{2})$ for distinct $q\in \Delta'$ are disjoint viewed as subsets of $PMF$. By equivalence of Hodge and Euclidean metrics there is a $h_{3}\in (0,h_{2})$ such that whenever $q\in \pi^{-1}K$ with $q'\in B^{su}_{E}(q,h_{3})$ 
we have $d_{H}(q,q')\leq h$.
For each $q\in \Delta'$ consider $$H(q)=\eta^{+}(B^{su}_{E}(q, h_{3}))\subset \eta^{+}(W(h)).$$
These are pairwise disjoint. Note, since $\nu_{x}$ has full support on $\Lambda(G)$, $\overline{\lambda}(H(q))>0$ for all $q\in Q^{1}WH(G)$. Thus, as  $\overline{\lambda}$ is $G$ equivariant and $G$ acts cocompactly on $Q^{1}WH(G)$, there is a $c=c(K,h)$ such that 
$$\overline{\lambda}(H(q))\geq c$$ for all $q\in Q^{1}WH(G)$.
Thus $$\overline{\lambda}(\eta^{+}(W(h)))\geq \sum_{q\in \Delta'} \overline{\lambda}(\eta^{+}(H(q)))\geq c|\Delta|.$$
This completes the proof.
\end{proof}

Now, let $P_{1}\subset Q^{1}(S)/G$ be compact (in our application $P_{1}$ we will be a subset of the principal stratum of almost full $\mu$ measure) and define $P_{3}\subset P_{2}\subset P_{1}$ and $\delta \in (0,1]$ such that if $q \in P_{i}$ and $d_{H}(q,q')\leq c_{H}h$ then $q'\subset P_{i-1}$ where $c_{H}$ is the nonexpansion constant of the modified Hodge norm over $P_{1}$. By choosing $h$ small enough we can assume $\mu(P_{3})>1/2$.
For $T_{0}>0$ let $U'_{i} = U'_{i}(T_{0})$ be the set of
$q\in Q^{1}WH(G)$ such that there exists $T>T_{0}$ so that $g_{t}q$ is in the complement of $P_{i}$ for at least half of $t \in [0,T]$.
By definition $U'_{1}\subset U'_{2}\subset U'_{3}$ and by the Birkhoff ergodic theorem, for every $\theta>0$ there is a $T_{0}>0$ such that $\mu(U'_{3})<\theta$.
Let $U_{i}=p^{-1}U_{i}\cap \pi^{-1}K$.
\begin{lem}
In the above notation, for all $t>0$
$$m(Nbhd_{2}(\pi(g_{t}U_{1})))\leq C(h)e^{\delta t}\overline{\lambda}(\eta^{+}(U_{2}))$$ and
 for any $\epsilon>0$ it is possible to choose $T_{0}$ such that for all $t>T_{0}$
$$m(Nbhd_{2}(\pi(g_{t}U_{1})))\leq \epsilon e^{\delta t}.$$
\end{lem}
\begin{proof}
Let $W=g_{t}U$. As shown in the proof of \cite{ABEM}, Lemma 4.2 we have $W(h)\subset g_{t}U_{2}$. Now, we can apply Lemma 11.9 to $W$ and use the fact that $$\overline{\lambda}((\eta^{+}g_{t}U))= e^{\delta t}\overline{\lambda}(\eta^{+}U)$$ to get the first claim. Moreover, as shown in the proof of [ABEM, Lemma 4.2], if $q\in U_{2}$ and $q'\in Q^{1}WH(G)$ is on the same strong stable leaf as $q$ with $d_{H}(q,q')<h$ then $q'\in U_{3}$. By compactness of $Q^{1}WH(G)/G$ and $G$ equivariance of $\mu^{ss}$, there is a $c>0$ such that $\mu^{ss}(B^{ss}(q))>c$ for all $q\in Q^{1}WH(G)\cap p^{-1}K$. Therefore, by the product structure of $\mu$,
$\lambda(U_{2})\leq C_{1}(h)\mu(U_{3})$ where $C_{1}(h)$ depends only on $h$.  Hence, choosing a large enough $T_{0}$ the second claim of the lemma follows. 
\end{proof}
\begin{proof}[Proof of Theorem 11.3]
In the above notation, let $P_{1}$ be chosen disjoint from the multiple zero locus. Let $T_{0}, U_{1}, U_{1}$ be as in the proof of Lemma 11.10. Let $K\subset Teich(S)$ be a fundamental domain for the action of $G$ on $WH(G)$ (so $m$ is supported on $G K=WH(G)$). Then for $R>T_{0}$ and $x \in K$ we have 
$$B_{R}(X,P_{1})\cap G K \subset \bigcup_{0\leq t\leq R} \pi (g_{t} U_{1})\cap G K \subset \bigcup^{\left\lfloor R\right\rfloor}_{n=0} \bigcup_{n\leq t\leq n+1} \pi (g_{t} U_{1})\cap G K$$
Then, $$m(Nbhd_{1} B_{R}(X,P))\leq \sum^{\left\lfloor R\right\rfloor}_{n=0} m(Nbhd_{2}(\pi (g_{n}U_{1}))\cap G K)\leq C \epsilon \sum^{\left\lfloor R\right\rfloor}_{n=0} e^{\delta n}$$ by Lemma 11.4. This completes the proof for $x\in WH(G)$.
\end{proof}

\section{Exact Asymptotics for Orbit Growth}
The goal of this section is to prove the part of Theorem 1.1 concerning orbit growth.
For $r>0$, $x\in Teich(S)$ and $A\subseteq PMF$ recall
$$C^{+}_{r}(x,A)=N_{r}\bigcup_{z\in B_{r}(x)}Sect_{z}(A)$$ and 
$$C^{-}_{r}(x,A)=\left\{y\in Teich(S)|B(y,r)\subseteq \bigcap_{z\in B(x,r)}Sect_{z}(A)\right\}$$

For $t>0$ and $x,y\in Teich (S)$ define a measure 
$$d\nu^{t}_{x,y}=\delta ||\mu|| e^{-\delta t}\sum_{d(x,\gamma y)\leq t}D_{\gamma x}\otimes D_{\gamma^{-1} y}$$
\begin{prop}
Let  $c>0$, $x,y\in Teich(S)$ and $\eta_{0},\zeta_{0}\in PMF$ be such that  there exist $\eta_{0}^{*},\zeta_{0}^{*}\in \Lambda(G)$ with $x\in (\eta_{0},\eta_{0}^{*})$, $y\in (\zeta_{0},\zeta_{0}^{*})$.
Then there exist open neighborhoods $V$ and $W$ in $PMF$ of $\eta_{0}$ and $\zeta_{0}$ respectively such that for all borel $A\subseteq V$ and $B\subseteq W$ with nonempty interior, as $T\to \infty$ we have 

$$\lim \sup \nu^{T}_{x,y}( C^{-}_{1}(x,A)\times C^{-}_{1}(y,B))\leq e^{c}\nu_{x}(A)\nu_{y}(B)$$ and
$$\lim \inf \nu^{T}_{x,y}( C^{+}_{1}(x,A)\times C^{+}_{1}(y,B))\geq e^{-c}\nu_{x}(A)\nu_{y}(B)$$
\end{prop}
\begin{proof}
If $\eta_{0}$ is not in $\Lambda(G)$ then we can choose a neighborhood $U$ of $\eta_{0}$ in $PMF$ with $\nu_{x}(U)=0$ and $W=PMF$ so that both sides of the desired equation are $0$. Similarly if $\zeta_{0}$ is not in $\Lambda(G)$.
Assume therefore that $\eta_{0},\zeta_{0}\in \Lambda(G)$. 
The argument is modelled on Roblin's Theorem 4.1.1 in \cite{Rob}, where an analogous result is proved for manifolds of pinched negative curvature.
For $\eta, \zeta \in PMF$ filling and $z\in Teich(S)$, let $z_{\eta,\zeta}$ be a quadratic differential with projective vertical and horizontal measured foliations $\zeta$ and $\eta$ respectively and such that $\pi(z_{\eta,\zeta})$ lies at minimal distance from $z$.
For $z\in Teich(S)$, $r>0$ and $A\subseteq PMF$ let 
$$K^{+}(z,r,A)=\left\{g_{s}z_{\eta,\zeta}|\eta \in A, d(z,(\eta,\zeta))<r, s\in [-\frac{r}{2},\frac{r}{2}]\right\}$$ and

$$K^{-}(z,r,A)=\left\{g_{s}z_{\eta,\zeta}|\zeta \in A, d(z,(\eta,\zeta))<r, s\in [-\frac{r}{2},\frac{r}{2}]\right\}$$

Let $$K(z,r)=K^{+}(z,r,PMF)=K^{-}(z,r,PMF)$$
Note that $\pi(K(z,r))\subseteq B(z, \frac{3r}{2})$.
For $a,b\in Teich(S)$ with $d(a,b)>2r$ let 
$$\Theta^{+}_{r}(a,b)= \bigcup_{w\in B(a,r)} pr_{w}(B(b,r))$$ and 
$$\Theta^{-}_{r}(a,b)= \bigcap_{w\in B(a,r)} pr_{w}(B(b,r))$$
Note, as $a\to \eta\in UE$ we have $\Theta^{\pm}_{r}(a,b)$ converging to $\Theta^{\pm}_{r}(\eta,b)=\Theta_{r}(\eta,b)=pr_{\eta}(B(b,r))$.

Let $L_{r}(a,b)\subseteq PMF\times PMF$ denote the pairs $(\eta,\zeta)$ such that the geodesic defined by them passes first through $B(a,r)$ and then $B(b,r)$.

It follows immediately from the definitions that  $$L_{r}(a,b)\subseteq \Theta^{+}_{r}(b,a)\times \Theta^{+}_{r}(a,b).$$

\begin{prop}
There exists an $r_{0}>0$ such that for all $0<r<r_{0}$ and all $h>0$ the following holds.
For each $\epsilon>0$ and $\epsilon'>0$ there exists an $R_{0}>0$ such that for every $a,b\in Teich(S)$ with the segment $[a,b]\subset Teich_{\epsilon}(S)$ and $d(a,b)>R_{0}$ such that  $[a,b]$ spends at least half the time in $Q_{\epsilon'}$ we have that $$\Theta^{-}_{r}(b,a)\times \Theta^{-}_{r}(a,b)\subset L_{r+h}(a,b)$$
\end{prop}
The proof of this proposition depends on the following lemmata from \cite{EM}, which say that geodesic segments that spend enough time in the thick part of stratum behave like geodesics in a CAT(-1) space.
\begin{lem}[\cite{EM}, Lemma 5.3]
Suppose $K\subseteq M_{g}$ is compact. Given $1>\beta>0$ there exists a $\rho_{0}>0$ (depending only on $K$ and $\beta$) with the following property.
Given $t>0$ and $\rho>0$ there exists an $L_{0}=L_{0}(K,t,\rho,\beta)$ such that if $X, p_{0}\in Teich(S)$ lie above $K$, $d_{T}(p_{0},p_{1})<\rho_{0}$, $d_{T}(X,p_{1})=L>L_{0}$, and 
$$|\left\{s\in [0,L]|l_{min}(g_{s}(q_{X,p_{0}}))\geq t\right\}|>\beta L$$ then
$$d_{E}(q, q_{X,p_{0}})<\rho$$ where $d_{E}$ denotes the Euclidean norm and $q$ is the unique quadratic differential in $W^{uu}(q_{X,p_{0}})\cap W^{s}(q_{X,p_{1}})$.
\end{lem}
By the equivalence of the Euclidean and Teichm$\ddot{\mathrm{u}}$ller metrics over compact subsets of $M_{g}$, we also have
 \begin{lem}
Suppose $K\subseteq M_{g}$ is compact. Given $1>\beta>0$ there exists a $\rho_{0}>0$ (depending only on $K$ and $\beta$) with the following property.
Given $t>0$ and $\rho>0$ there exists an $L_{0}=L_{0}(K,t,\rho,\beta)$ such that if $X, p_{0}\in Teich(S)$ lie above $K$, $d_{T}(p_{0},p_{1})<\rho_{0}$, $d_{T}(X,p_{1})=L>L_{0}$, and 
$$|\left\{s\in [0,L]|l_{min}(g_{s}(q_{X,p_{0}}))\geq t\right\}|>\beta L$$ then
$$d_{T}(\pi(q), X)<\rho$$ where $d_{E}$ denotes the Euclidean norm and $q$ is the unique quadratic differential in $W^{uu}(q_{X,p_{0}})\cap W^{s}(q_{X,p_{1}})$.
\end{lem}
\begin{lem}[\cite{EM}, Lemma 5.4] Suppose $K\subset M_{g}$ is compact. Given $s>0$, there exists  constants
$L_{0}$ depending on $s$ and $K$, and $c_{0}$ depending only on $K$ with the following property. If 
$\gamma:[0,L]\to Q^{1}(S)$ is a geodesic segment (parametrized by arclength)  with endpoints above $K$ and $L>L_{0}$, $\widehat{\gamma}:[0,L']\to Q^{1}(S)$ is the geodesic segment connecting $p_{1},p_{2}\in Teich(S)$ such that $d_{T}(p_{1},\pi(\gamma(0)))<c_{0}$, $d_{T}(p_{2},\pi(\gamma(L)))<c_{0}$ and
$$|\left\{s\in [0,L]|l_{min}(\gamma(t))\geq s\right\}|>\frac{ L}{2}$$ then 
$$|\left\{s\in [0,L']|l_{min}(\widehat{\gamma}(t))\geq s/4\right\}|>\frac{ L'}{3}$$
\end{lem}
\makefig{proof of Proposition  12.2}{fig:12.2}{\includegraphics{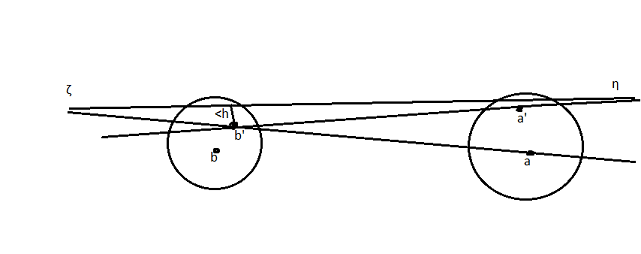}}
\begin{proof}[Proof of Proposition 12.2]

Let $K=Teich_{\epsilon}(S)/Mod(S)$. Let $L_{1}$ (depending on $K$ and $\epsilon'$) and $c_{0}$ (depending on $K$) be the $L_{0}$ in  Lemma 12.4 corresponding to $s=\epsilon'$ and $K=Teich_{\epsilon}(S)/Mod(S)$. Then, if $(x,y)\subseteq Teich_{\epsilon}(S)$ and $d_{T}(x,y)=L>L_{0}$ with $$|\left\{s\in [0,L]|l_{min}(g_{t}q_{x,y})\geq \epsilon'\right\}|>\frac{ L}{2}$$ it follows from Lemma 12.5 that for every $p_{1}\in B_{c_{0}}(x)$ and $p_{2}\in B_{c_{0}}(y)$ we have that 
$$|\left\{s\in [0,L]|l_{min}(g_{t}q_{x,y})\geq \epsilon'/4\right\}|>\frac{ L}{3}$$
Now let $\rho_{0}$ be as in Lemma 12.3, corresponding to $K=Teich_{\epsilon}(S)/Mod(S)$ and $\beta=1/3$.
Let $r_{0}=min\left\{c_{0},\rho_{0}\right\}$. Suppose $r<r_{0}$ and $h>0$ arbitrary.
 Let $L_{2}$ be the $L_{0}$ in Lemma 12.4 corresponding to $K=Teich_{\epsilon}(S)/Mod(S)$ and $\beta=1/3$, $t=\epsilon'/4$, and $\rho=h$. Let $R_{0}>max\left\{L_{1}+2r_{0},L_{2}+2r_{0}\right\}$. Suppose  $a,b\in Teich(S)$ with the segment $[a,b]\subset Teich_{\epsilon}(S)$ and $d(a,b)>R_{0}$ such that  $[a,b]$ spends at least half the time outside of $Q_{\epsilon'}$ and $\eta\in \Theta^{-}_{r}(b,a)$, $\zeta\in \Theta^{-}_{r}(a,b)$. Let $b'\in B_{r}(b)$ be a point on the geodesic containing $[a,\zeta)$ and $a'\in B_{r}(a)$ a point on the geodesic containing $[b',\eta)$.
We will apply Lemma 12.4 with $X=b'$, $p_{0}=a'$, $p_{1}=a$ to obtain that the quadratic differential $q\in W^{uu}(q_{b',\eta})\cap W^{s}(q_{b',\zeta})$ satisfies $d_{T}(\pi(q),b')<h$, so $\pi(q)\in B_{r+h}(b)$. By definition $\pi(q)\in (\eta,\zeta)$ so $(\eta,\zeta)$ intersects $B_{r+h}(b)$. Similarly, $(\eta,\zeta)$ intersects $B_{r+h}(a)$ completing the proof.
\end{proof}
We now continue with the proof of Proposition 12.1.
From now on, fix $r\in [0, \frac{c}{120\delta}]$ smaller than the $r_{0}$ in Proposition 12.2 such that $$\nu_{x}(\partial{\Theta_{r}(\eta_{0},x)})=\nu_{y}(\partial{\Theta_{r}(\zeta_{0},y)})=0$$ (this last condition only excludes countably many values of $r$). Since $\eta_{0}^{*}\in \Theta_{r}(\eta_{0},x)$ and $\zeta_{0}^{*}\in \Theta_{r}(\zeta_{0},y)$ and the conformal densities have full support on $\Lambda(G)$ we have that $$\nu_{x}(\Theta_{r}(\eta_{0},x)\nu_{y}(\Theta_{r}(\eta_{0},y)>0$$
Now, fix $h>0$ such that $$\nu_{x}(\Theta_{r-h}(\eta_{0},x))\geq e^{-c/120}\nu_{x}(\Theta_{r}(\eta_{0},x))$$
and
$$\nu_{y}(\Theta_{r-h}(\zeta_{0},y))\geq e^{-c/120}\nu_{y}(\Theta_{r}(\zeta_{0},y))$$ and also
$$\nu_{x}(\partial{\Theta_{r-h}(\eta_{0},x)})=\nu_{y}(\partial{\Theta_{r-h}(\zeta_{0},y)})=0$$
Let  $\widehat{V}$ and $\widehat{W}$ be open neighborhoods in $Teich(S)\cup PMF$ of $\eta_{0}$ and $\zeta_{0}$ respectively such that for all $(a,b)\in (\widehat{V}\times \widehat{W})$ with $a,b\in N_{D}WH(G)\cup \Lambda(G)$ with $D=d(x,WH(G))+ d(y,WH(G))+1$ we have
$$e^{-c/120}\nu_{x}(\Theta_{r}(\eta_{0},x))\leq\nu_{x}(\Theta^{\pm}_{r}(a,x))\leq e^{c/120}\nu_{x}(\Theta_{r}(\eta_{0},x))$$ and 
$$e^{-c/120}\nu_{y}(\Theta_{r}(\zeta_{0},y))\leq\nu_{y}(\Theta^{\pm}_{r}(b,y))\leq e^{c/120}\nu_{y}(\Theta_{r}(\zeta_{0},y))$$

$$e^{-c/120}\nu_{x}(\Theta_{r-h}(\eta_{0},x))\leq\nu_{x}(\Theta^{\pm}_{r-h}(a,x))\leq e^{c/120}\nu_{x}(\Theta_{r-h}(\eta_{0},x))$$ and 
$$e^{-c/120}\nu_{y}(\Theta_{r-h}(\zeta_{0},y))\leq\nu_{y}(\Theta^{\pm}_{r-h}(b,y))\leq e^{c/120}\nu_{y}(\Theta_{r-h}(\zeta_{0},y))$$

It then follows that
$$e^{-c/60}\nu_{x}(\Theta_{r}(\eta_{0},x))\leq\nu_{x}(\Theta^{\pm}_{r-h}(a,x))\leq\nu_{x}(\Theta^{\pm}_{r}(a,x))\leq e^{c/60}\nu_{x}(\Theta_{r}(\eta_{0},x))$$ and 
$$e^{-c/60}\nu_{y}(\Theta_{r}(\zeta_{0},y))\leq\nu_{y}(\Theta^{\pm}_{r-h}(b,y))\leq\nu_{y}(\Theta^{\pm}_{r}(b,y))\leq e^{c/120}\nu_{y}(\Theta_{r}(\zeta_{0},y))$$

Let $V$ and $W$ be open neighborhoods of $\eta_{0}$ and $\zeta_{0}$ respectively in $PMF$ such that $\overline{V}\subseteq \widehat{V}\cap PMF$ and 
$\overline{W}\subseteq \widehat{W}\cap PMF$.
Consider open  subsets $A\subseteq V$ and $B\subseteq W$. Let 
$$K^{+}=K^{+}(x,r,A)$$ and
$$K^{+}=K^{+}(y,r,B)$$
We will estimate  as $T\to \infty$ the quantity 
$$\int^{T}_{0} e^{\delta t}\sum_{\gamma \in G}\mu(K^{+}\cap g_{-t}\gamma K^{-})dt$$

From the definitions, it follows that for $\gamma \in G$ and for $d(x,\gamma y)> 2r$ we have
$$\mu(K^{+}\cap g_{-t}\gamma K^{-})dt=\int_{L_{r}(x,\gamma y)\cap (\gamma B \times A)}e^{\delta \rho_{x}(\eta, \zeta)}d\nu_{x}(\eta)d\nu_{y}(\zeta) \int^{r/2}_{-r/2} \chi_{K(\gamma y, r)} (g_{t+s}x_{\eta,\zeta})ds$$
We first find an upper bound for 
$$\int^{T-3r}_{0} e^{\delta t}\sum_{\gamma \in G}\mu(K^{+}\cap g_{-t}\gamma K^{-})dt$$
First, note that for $(\eta,\zeta)\in L_{r}(x,y)$ we have $$\rho_{x}(\eta,\zeta)\leq 2r \leq c/30$$
Now, suppose $L_{r}(x,\gamma y)\cap (\gamma B \times A)$ is nonempty. Then from the definitions it follows that
$\gamma y \in C^{+}_{r}(x,A)\subseteq C^{+}_{1}(x,A)$. Since $\gamma$ is an isometry, it also holds that 
$L_{r}(\gamma^{-1} x, y)\cap (B \times \gamma^{-1} A)$ is nonempty and thus $\gamma^{-1}x\in C^{+}_{1}(y,b)$.

Note that for $(\eta,\zeta)\subseteq L_{r}(x,y)$, $|s|<r/2$ and $T>0$ we have 
$$\int ^{T-3r}_{0}e^{\delta t}\chi_{K(\gamma y, r)} (g_{t+s}x_{\eta,\zeta})dt\leq e^{\delta (3r)}r e^{\delta d(x, \gamma y)}\leq e^{c/20}r e^{\delta d(x, \gamma y)}$$ 

and moreover is zero whenever $d(x,\gamma y)>T$.
Using the fact that $L_{r}(a,b)\subseteq \Theta^{+}_{r}(b,a)\times \Theta^{+}_{r}(a,b)$ it follows that $$\int^{T-3r}_{0} e^{\delta t}\sum_{\gamma \in G}\mu(K^{+}\cap g_{-t}\gamma K^{-})dt\leq
e^{c/12}r^{2}\sum \nu_{x}(\Theta^{+}_{r}(\gamma y,x))\nu_{y}(\Theta^{+}_{r}(x, \gamma y))$$ where the sum is taken over all $\gamma \in G$ such that $(x,\gamma y)\leq T$ and $$(\gamma y, \gamma^{-1}x)\in [C^{+}_{1}(x,A)\times C^{+}_{1}(y,B)]$$ 
By Corollary 7.4 we see that he set
$$[C^{+}_{1}(x,A)\cap N_{D}WH(G)\times C^{+}_{1}(y,B)\cap N_{D}WH(G)]\setminus [\widehat{V}\times \widehat{W}]$$ has compact closure in $Teich(S)$ for any $D>0$.
If $x,y\in N_{D}WH(G)$ their $G$ orbits are also contained in $N_{D}WH(G)$. Thus, by the discreteness of the $G$ action on $Teich(S)$, for some constant $c_{1}$ that does not depend on $T$
$$\int^{T-3r}_{0} e^{\delta t}\sum_{\gamma \in G}\mu(K^{+}\cap g_{-t}\gamma K^{-})dt\leq e^{c/12}r^{2}\sum \nu_{x}(\Theta^{+}_{r}(\gamma y,x))\nu_{y}(\Theta^{+}_{r}(x, \gamma y))-c_{1}$$ 

for all  $T>0$
 where the sum is taken over all $\gamma \in G$ with $(x,\gamma y)\leq T$ and $(\gamma y, \gamma^{-1}x)\in [C^{+}_{1}(x,A)\times C^{+}_{1}(y,B)]\cap [\widehat{V}\times \widehat{W}]$.

Note, by the triangle inequality we have for $\eta\in \Theta^{+}_{r}(x,\gamma y))$ 
$$d(x, \gamma y)-4r \leq \beta_{\eta}(x,\gamma y)\leq d(x,\gamma y)$$
and by the conformality of $\nu$  it follows that
$$\nu_{y}(\Theta^{+}_{r}(\gamma^{-1} x,y))=\nu_{\gamma y}(\Theta^{+}_{r}(x,\gamma y))\leq \nu_{x}(\Theta^{+}_{r}(x,\gamma y))e^{\delta d(x,\gamma y)}\leq e^{4r}\nu_{\gamma y}(\Theta^{+}_{r}(x,\gamma y))$$

Since $4\delta r<c/15$ it follows that when $$(\gamma y, \gamma^{-1}x)\subseteq \widehat{V}\times \widehat{W}$$ we have that
$$\nu_{x}(\Theta^{+}_{r}(\gamma y,x))\nu_{x}(\Theta^{+}_{r}(x,\gamma y))\leq e^{c/15}\nu_{x}(\Theta^{+}_{r}(\gamma y,x))\nu_{y}(\Theta^{+}_{r}(\gamma^{-1} x,y))$$
$$\leq e^{c/10}\nu_{x}(\Theta_{r}(\eta_{0},x))\nu_{y}(\Theta_{r}(\zeta_{0},y))$$

Thus we obtain 

$$\int^{T-3r}_{0} e^{\delta t}\sum_{\gamma \in G}\mu(K^{+}\cap g_{-t}\gamma K^{-})dt\\\leq e^{c/6}r^{2}|G^{+}(T,A,B)| \nu_{x}(\Theta_{r}(\eta_{0},x))\nu_{y}(\Theta_{r}(\zeta_{0},y))+c_{1}$$
where $c_{1}$ is independent of $T$ and $G^{+}(T,A,B)$ is the set of all $\gamma \in G$ such that $$(x,\gamma y)\leq T$$ and $$(\gamma y, \gamma^{-1}x)\in [C^{+}_{1}(x,A)\times C^{+}_{1}(y,B)]\cap [\widehat{V}\times \widehat{W}]$$

In a similar but more annoying manner, we will obtain a lower bound for 
$$\int^{T+3r}_{0} e^{\delta t}\sum_{\gamma \in G}\mu(K^{+}\cap g_{-t}\gamma K^{-})dt$$

First, note that for $\eta \in \Theta^{-}_{r-h}(a,b)$ we have $$d(a,b)-2r \leq\beta_{\eta}(a,b)\leq d(a,b)$$
and thus similarly to above we have 
$$\nu_{y}(\Theta^{-}_{r-h}(\gamma^{-1} x,y))=\nu_{\gamma y}(\Theta^{-}_{r-h}(x,\gamma y))\leq \nu_{x}(\Theta^{-}_{r-h}(x,\gamma y))e^{\delta d(x,\gamma y)}\leq e^{2r}\nu_{\gamma y}(\Theta^{-}_{r-h}(x,\gamma y))$$

Since $2\delta r<c/30$ it follows that when $$(\gamma y, \gamma^{-1}x)\subseteq \widehat{V}\times \widehat{W}$$ we have that

$$\nu_{x}(\Theta^{-}_{r-h}(x,\gamma y))\nu_{x}(\Theta^{-}_{r-h}(\gamma y,x))\geq e^{-\delta d(x,\gamma y)}\nu_{x}(\Theta^{-}_{r-h}(x,\gamma y))\nu_{y}(\Theta^{-}_{r-h}(x,\gamma y))$$
$$\geq e^{-c/30}\nu_{x}(\Theta_{r}(\eta_{0},x))\nu_{y}(\Theta_{r}(\zeta_{0},y))$$
\\ 
Now note, if $(\gamma y, \gamma^{-1}x)\in C^{-}_{1}(x,A)\times C^{-}_{1}(y,B)$ then by definition

$A \supset \Theta^{-}_{r}(x,\gamma y)$ and $B\supset \Theta^{-}_{r}(y,\gamma^{-1} x)$, whence $\gamma B\supset \Theta^{-}_{r}(\gamma y,x)$.
Note for $(\eta,\zeta)\in L_{r}(x,y), |s|<r/2$, $T>0$ and $3r\leq d(x,\gamma y)\leq T$ we have
$$\int ^{T+3r}_{0}e^{\delta t}\chi_{K(\gamma y, r)} (g_{t+s}x_{\eta,\zeta})dt\geq e^{-3\delta r}r e^{\delta d(x,\gamma y)}\geq e^{-c/20}r e^{\delta d(x,\gamma y)}$$

Now fix an $\epsilon'>0$ with $\mu^{BMS}(Q_{\epsilon'}(S)/G)<1/3$ and consider $\gamma \in G$ such that $[x,\gamma y]$ and $[\gamma y, x]$ both spend less than half time in $Q_{\epsilon'}$. By Proposition 12.6 and the discreteness of the action of $G$, for all but finitely many such $\gamma$ we have that $$\Theta^{-}_{r-h}(\gamma y, x)\times \Theta^{-}_{r-h}(x,\gamma y)\subset L_{r}(x,\gamma y)$$

Note for $(\eta,\zeta)\in L_{r}(x,y), |s|<r/2$, $T>0$ and $3r\leq d(x,\gamma y)\leq T$ we have
$$\int ^{T+3r}_{0}e^{\delta t}\chi_{K(\gamma y, r)} (g_{t+s}x_{\eta,\zeta})dt\geq e^{-c/20}r e^{\delta d(x,\gamma y)}$$

Thus we have that  

$$\int^{T+3r}_{0} e^{\delta t}\sum_{\gamma \in G}\mu(K^{+}\cap g_{-t}\gamma K^{-})dt\geq e^{-c/20}r^{2}\sum_{\gamma \in G(T,A,B)}{\nu_{x}(\Theta^{-}_{r-h}(\gamma y,x))\nu_{x}(\Theta^{-}_{r-h}(x,\gamma y))e^{\delta d(x,\gamma y)}}\geq$$ $$ e^{-c/12}r^{2}|G^{-}(T,A,B)\setminus G^{-}(T,\epsilon', A,B)| \nu_{x}(\Theta_{r}(\eta_{0},x)) \nu_{y}(\Theta_{r}(\zeta_{0},y))-c_{2}$$
where  $G^{-}(T, A,B)$ is the set of all $\gamma \in G$ such that $(x,\gamma y)\leq T$ and $$(\gamma y, \gamma^{-1}x)\in [C^{+}_{1}(x,A)\times C^{+}_{1}(y,B)]\cap [\widehat{V}\times \widehat{W}]$$
 and $G^{-}(T,\epsilon', A,B)$  is the set of all $\gamma \in G(T,A,B)$ such that 
the segment $(x,\gamma y)$ spends at least half the time in the $\epsilon'$ thin part of the principal stratum, and $c_{2}$ does not depend on $T$.

By mixing it follows that for all $t$ large enough we have 
$$e^{-c/60}\mu(K^{+})\mu(K^{-})\leq ||\mu||\sum_{\gamma \in G} \mu(K^{+}\cap g^{-t}\gamma K^{-})\leq e^{c/60}\mu(K^{+})\mu(K^{-})$$
Note, that by definition 
$$ \mu(K^{+})=r\int_{\eta \in A}\int_{\zeta \in \Theta_{r}(\eta,x)}e^{\delta \rho_{x}(\eta,\zeta)}d\nu_{x}(\zeta)d\nu_{x}(\eta)$$
Since $$0\leq \rho_{x}(\eta,\zeta)\leq 2r$$ for $\zeta \in \Theta_{r}(\eta,x)$ and since $A\subseteq V$
we obtain that $$e^{-c/60}r\nu_{x}(A)\nu_{x}(\Theta_{r}(\eta_{0},x))\leq \mu(K^{+})\leq e^{c/20}r\nu_{x}(A)\nu_{x}(\Theta_{r}(\eta_{0},x))$$ and similarly $$e^{-c/60}r\nu_{y}(A)\nu_{y}(\Theta_{r}(\zeta_{0},y))\leq \mu(K^{-})\leq e^{c/20}r\nu_{y}(A)\nu_{y}(\Theta_{r}(\zeta_{0},y))$$

It follows that there exists a constant $c_{2}$ independent of $T$ such that
$$\delta ||\mu||\int^{T-3r}_{0} e^{\delta t}\sum_{\gamma \in G}\mu(K^{+}\cap g_{-t}\gamma K^{-})dt\geq e^{-c/2}e^{\delta T} M\nu_{x}(A)\nu_{y}(B)-c_{2}$$ and  
$$\delta ||\mu||\int^{T+3r}_{0} e^{\delta t}\sum_{\gamma \in G}\mu(K^{+}\cap g_{-t}\gamma K^{-})dt\leq e^{c/2}e^{\delta T} M\nu_{x}(A)\nu_{y}(B)+c_{2}$$ where $M=r^{2}\nu_{x}(\Theta_{r}(\eta_{0},x))\nu_{y}(\Theta_{r}(\zeta_{0},y))$.

Thus, it follows that $$e^{-c/2}\nu_{x}(A)\nu_{x}(B)\leq e^{c/3}e^{-\delta T}|G^{+}(T,A,B)|$$ and $$e^{c/2}\nu_{x}(A)\nu_{x}(B)\geq e^{-c/6}e^{-\delta T}|G^{-}(T,A,B) \setminus G^{-}(T,\epsilon' A,B)|$$ for large enough $T$.
Furthermore, by Theorem 11.4, if $\epsilon'$ is chosen small enough,
$$\lim \sup_{T\to \infty} |G^{-}(T,\epsilon' A,B)|/e^{\delta t}<e^{c/12}$$ so that
$$e^{c}\nu_{x}(A)\nu_{x}(B)\geq e^{-\delta T}|G^{-}(T,A,B)|$$

for all large enough $T$.

This completes the proof of Proposition 12.1.
\end{proof}

\begin{lem}
Let $x,y\in Teich(S)$ and $c>0$. For each $$(\eta_{0},\zeta_{0})\in PMF \times PMF$$ there exists an $r>0$ and neighborhoods $V$ and $W$ of $\eta_{0}$ and $\zeta_{0}$ in $PMF$ respectively such that for all borel $A\subseteq V$ and $B\subseteq W$, with nonempty interior: 
$$\lim \sup_{t \to \infty} \nu^{t}_{x,y}( C^{-}_{r}(x,A)\times C^{-}_{r}(y,B))\leq e^{c}\nu_{x}(A)\nu_{y}(B)$$ and
$$\lim \inf_{t\to \infty} \nu^{t}_{x,y}( C^{+}_{r}(x,A)\times C^{+}_{r}(y,B))\geq e^{-c}\nu_{x}(A)\nu_{y}(B)$$
\end{lem}

\begin{proof}
If $\eta_{0}$ is not in $\Lambda(G)$ then we can choose a neighborhood $U$ of $\eta_{0}$ in $PMF$ with $\nu_{x}(U)=0$ and $W=PMF$ so that both sides of the desired equation are $0$ by Corollary 9.7. Similarly if $\zeta_{0}$ is not in $\Lambda(G)$.
Assume therefore that $\eta_{0},\zeta_{0}\in \Lambda(G)$. 
Let $\lambda_{0}\in \Lambda(G)$ and $x_{0}\in (\eta_{0},\lambda_{0})$, $y_{0}\in (\zeta_{0},\lambda_{0})$. Let $V_{0}, W_{0}$ be open neighborhoods of $\eta_{0}$ and $\zeta_{0}$ in $PMF$ respectively such that for all open $A\subseteq V_{0}$ and $B\subseteq W_{0}$, we have as $T\to \infty$ that  
$$\lim \sup \nu^{T}_{x_{0},y_{0}}( C^{-}_{1}(x_{0},A)\times C^{-}_{r}(y_{0},B))\leq e^{c/3}\nu_{x_{0}}(A)\nu_{y_{0}}(B)$$ and
$$\lim \inf \nu^{T}_{x_{0},y_{0}}( C^{+}_{r}(x_{0},A)\times C^{+}_{r}(y_{0},B))\geq e^{-c/3}\nu_{x_{0}}(A)\nu_{y_{0}}(B)$$ 
Let $\widehat{V_{0}}$ and $\widehat{W_{0}}$ be neighborhoods in $Teich(S)\cup PMF$ of $\eta_{0}$ and $\zeta_{0}$ respectively, whose intersection with $PMF$ are respectively contained in $V_{0}$ and $W_{0}$ and such that for all $$a\in \widehat{V_{0}}\cap N_{D}WH(G), b\in \widehat{W_{0}}\cap N_{D}WH(G)$$ we have
$$|d(x_{0},a)-d(x,a)-\beta_{\eta_{0}}(x_{0},x)|\leq \frac{c}{6\delta}$$ 
$$|d(y_{0},b)-d(y,b)-\beta_{\eta_{0}}(y_{0},y)|\leq \frac{c}{6\delta}$$ 
and for all $$\eta \in \widehat{V_{0}} \cap \Lambda(G)$$ $$\zeta \in \widehat{W_{0}} \cap \Lambda(G)$$
we have $$|\beta_{\eta}(x_{0},x)-\beta_{\eta_{0}}(x_{0},x)|\leq \frac{c}{6\delta}$$ and 
$$|\beta_{\eta}(x_{0},x)-\beta_{\eta_{0}}(x_{0},x)|\leq \frac{c}{6\delta}.$$

Let $V$ and $W$ be open neighborhoods in $PMF$ of $\eta_{0}$ and $\zeta_{0}$ respectively, with $\overline{V}\subseteq \widehat{V_{0}}\cap PMF$ and $\overline{W}\subseteq \widehat{W_{0}}\cap PMF$ and let $$r=1+d(x,x_{0})+d(y,y_{0}).$$ Consider $A\subseteq V$ and $B\subseteq W$.

Note, if $(\gamma y, \gamma^{-1} x)\in C^{-}_{r}(x,A)\times C^{-}_{r}(y,B)$ one easily checks that $(\gamma y_{0}, \gamma^{-1} x_{0})\in C^{-}_{1}(x_{0},A)\times C^{-}_{1}(y_{0},B)$ by the choice of $r$. 
Next, note that if $d(x,\gamma y)\leq t$ and $(\gamma y, \gamma^{-1}x)\in \widehat{V_{-r}}\times \widehat W_{0}$ where $\widehat{V_{-r}}$ denotes the set of points whose $r$ neighborhood is contained in $\widehat{V_{0}}$, then $\gamma y_{0}\in \widehat{V_{0}}$ and $\gamma^{-1}x\in \widehat{W_{0}}$ which implies that
$$d(x_{0},\gamma y_{0})\leq d(x,\gamma y_{0})+\beta_{\eta_{0}}(x_{0},x)+\frac{c}{6\delta}=
d(y_{0},\gamma^{-1}x)+\beta_{\eta_{0}}(x_{0},x)+\frac{c}{6\delta}$$
$$\leq d(y,\gamma^{-1}x)+\beta_{\zeta_{0}}(y_{0},y)+\beta_{\eta_{0}}(x_{0},x)+\frac{c}{3\delta}\leq t+\beta_{\eta_{0}}(x_{0},x)+\beta_{\zeta_{0}}(y_{0},y)+\frac{c}{3\delta}$$
From Corollary 7.4 we obtain
\begin{lem}
$$[C^{-}_{r}(x,A)\cap N_{D}WH(G)\times C^{-}_{r}(y,B)\cap N_{D}WH(G)]\setminus [\widehat{V_{-r}}\times \widehat{W_{0}}]$$ is  relatively compact in $Teich(S)$.
\end{lem}
 Thus $[C^{-}_{r}(x,A)\times C^{-}_{r}(y,B)]\setminus [\widehat{V_{-r}}\times \widehat{W_{0}}]$ contains only finitely many points $(\gamma x,\gamma^{-1}y)$, $\gamma \in G$

From this, we deduce that  $$\lim \sup_{t\to\infty} \nu^{t}_{x,y}( C^{-}_{r}(x,A)\times C^{-}_{r}(y,B))\leq$$ 
$$e^{c/3}e^{\delta \beta_{\eta_{0}}(x_{0},x)+\delta \beta_{\zeta_{0}}(y_{0},y)}\lim \sup \nu_{x_{0},y_{0}}^{t+\beta_{\eta_{0}}(x_{0},x)+\beta_{\zeta_{0}}(y_{0},y)+\frac{c}{3\delta}}(C^{-}_{1}(x_{0},A)\times C^{-}_{1}(y_{0},B))$$

and thus by Prop 12.1, 
$$\lim \sup \nu^{t}_{x,y}( C^{-}_{r}(x,A)\times C^{-}_{r}(y,B))\leq e^{2c/3}e^{\delta \beta_{\eta_{0}}(x_{0},x)+\delta \beta_{\zeta_{0}}(y_{0},y)}\nu_{x_{0}}(A)\nu_{y_{0}}(B).$$
Since $e^{\delta \beta_{\eta_{0}}(x_{0},x)}\nu_{x_{0}}\leq e^{c/6}\nu_{x}$ when restricted to $V$ and 
$e^{\delta \beta_{\eta_{0}}(y_{0},y)}\nu_{y_{0}}\leq e^{c/6}\nu_{y}$ when restricted to $W$, we obtain that 
$$\lim \sup \nu^{t}_{x,y}(C^{-}_{r}(x.A)\times C^{-}_{r}(y.B))\leq e^{c}\nu_{x}(A)\nu_{y}(B)$$
The reverse estimate is proved similarly.
\end{proof}

\begin{thm}
For $x,y\in Teich(S)$
$\nu^{t}_{x,y}$ converges weakly to $\nu_{x}\times \nu_{y}$ as $t\to \infty$
\end{thm}
\begin{proof}
Let $c>0$.  For each $(\eta_{0},\zeta_{0})\in \Lambda(G)\times \Lambda(G)$ take  neighborhoods $V_{(\eta_{0},\zeta_{0})}$ and $W_{(\eta_{0},\zeta_{0})}$ of $\eta_{0}$ and $\zeta_{0}$ respectively such that the conclusion of Lemma 12.6 holds for $c$. By compactness finitely many of the $V\times W$ cover $PMF\times PMF$, say $V_{i}\times W_{i}$, $i=1,...,n.$
Let $\widehat{V_{i}}$ and $\widehat{W_{i}}$ be open subsets of $Teich(S)\cup PMF$ such that $V_{i}= \widehat{V_{i}}\cap PMF$ and $W_{i}= \widehat{W_{i}}\cap PMF$.
Let $\widehat{A}$ and  $\widehat{B}$ be borel subsets of $Teich(S)\cup PMF$ with
$$\overline{\widehat{A}}\subset \widehat{V_{i}}$$ $$\overline{\widehat{B}}\subset \widehat{W_{i}}$$ and

 $$(\nu_{x}\otimes\nu_{y})(\partial(\widehat{A}\times\widehat{B}))=0$$ ie $$\nu_{x}(\overline{\widehat{A}})) \nu_{y}(\partial{\widehat{B}})=\nu_{x}(\overline{\widehat{B}}))\times \nu_{y}(\partial{\widehat{A}})=0$$ 
Let $\alpha>0$.
Let $A^{+},B^{+}\subset PMF$ be open and $A^{-},B^{-}\subset PMF$ compact with $A^{-},B^{-}$ either being empty or having nonempty interior such that 
$$A^{-}\subset \widehat{A}^{o}\cap PMF\subset \overline{\widehat{A}}\cap PMF\subset A^{+}\subset V_{i}$$ 

$$B^{-}\subset \widehat{B}^{o}\cap PMF\subset \overline{\widehat{B}}\cap PMF\subset B^{+}\subset W_{i}$$
$$\nu_{x}(\widehat{A}^{o}\setminus A^{-})<\alpha, \nu_{x}(A^{+}\setminus  \overline{\widehat{A}})<\alpha,  \nu_{x}(\widehat{B}^{o}\setminus B^{-})<\alpha, \nu_{x}(A^{+}\setminus  \overline{\widehat{A}})<\alpha$$
Let $D>d(x,WH(G))+d(y,WH(G))$ so that the $G$ orbits of $x,y$ are contained in $N_{D}WH(G)$.
By Corollary 7.4 the sets
$$[\overline{\widehat{A}}\cap N_{D}WH(G)\times \overline{\widehat{B}}\cap N_{D}WH(G)]\setminus [C^{-}_{r}(x,A^{+})\times C^{-}_{r}(y,B^{+})]$$ and
$$[C^{+}_{r}(x,A^{-})\cap N_{D}WH(G)\times C^{+}_{r}(y,B^{-})\cap N_{D}WH(G)]\setminus [\widehat{A}^{o}\times\widehat{B}^{o}]$$
are relatively compact in $Teich(S)\times Teich(S)$.
Thus, by Lemma 12.6 we have 
$$\lim \sup \nu^{t}_{x,y}(\widehat{A}\times \widehat{B})\leq \lim \sup \nu^{t}_{x.y}(C^{-}_{r}(x,A^{+})\times C^{-}_{r}(y,B^{+}))\leq e^{c}\nu_{x}(A^{+})\nu_{y}(B^{+})\leq$$ $$e^{c}\nu_{x}(\overline{\widehat{A}})\nu_{y}(\overline{\widehat{B}})+\alpha e^{c}(||\nu_{x}||+||\nu_{y}||)=e^{c}\nu_{x}(\widehat{A})\nu_{y}(\widehat{B})+\alpha e^{c}(||\nu_{x}||+||\nu_{y}||)$$
Since $\alpha>0$ can be chosen arbitrarily small we obtain 
$$\lim \sup \nu^{t}_{x,y}(\widehat{A}\times \widehat{B})\leq e^{c}\nu_{x}(\overline{A})\nu_{y}(\overline{B})$$
Similarly we obtain the reverse estimate 
$$\lim \inf \nu^{t}_{x,y}(\widehat{A}\times \widehat{B})\geq e^{-c}\nu_{x}(\overline{A})\nu_{y}(\overline{B})$$
Indeed, $$\lim \inf \nu^{t}_{x,y}(\widehat{A}\times \widehat{B})\geq \lim \inf \nu^{t}_{x.y}(C^{+}_{r}(x,A^{-})\times C^{+}_{r}(y,B^{-}))\geq$$ $$e^{-c}\nu_{x}(\widehat{A}^{o})\nu_{y}(\widehat{B}^{o})-\alpha e^{c}(||\nu_{x}||+||\nu_{y}||)=e^{-c}\nu_{x}(\widehat{A})\nu_{y}(\widehat{B})-\alpha e^{c}(||\nu_{x}||+||\nu_{y}||)$$ for any $\alpha>0$.
Thus, for $\phi$ a continuous function supported on $\widehat{V}\times \widehat{W}$ we have $$e^{-c}\int \phi d\nu_{x}\otimes d\nu_{y}\leq \lim \inf \int \phi d\nu^{t}_{x,y}\leq  \lim \inf \int \phi d\nu^{t}_{x,y}\leq e^{-c}\int \phi d\nu_{x}\otimes d\nu_{y}$$
Furthermore the complement $O$ of $\bigcup^{n}_{i=1} V_{i}\times W_{i}$ in $Teich(S)\cup PMF$ is a compact subset of $Teich(S)$, so for any function supported on $O$ we have 

$$\int \phi d\nu_{x}\otimes d\nu_{y}=0$$ and $$\lim_{t\to \infty} \int \phi d\nu^{t}_{x,y}=0$$

By choosing a partition of unity subordinate to the cover $O,\widehat{V_{i}}\times \widehat{W_{i}}$ we obtain that
$$e^{-c}\int \phi d\nu_{x}\otimes d\nu_{y}\leq \lim \inf \int \phi d\nu^{t}_{x,y}\leq  \lim \inf \int \phi d\nu^{t}_{x,y}\leq e^{-c}\int \phi d\nu_{x}\otimes d\nu_{y}$$ for any continuous $\phi$ on $(Teich(S)\cup PMF)\times (Teich(S)\cup PMF)$ and letting $c\to 0$ yields the desired result.

\end{proof}
Theorem 1.1 follows.

\begin{thm}
For all $x,y\in Teich(S)$ we have
$$\lim_{R\to \infty}|B_{R}(x)\cap Gy|e^{-\delta R}=\frac{||\nu_{x}|| ||\nu_{y}||}{\delta ||\mu^{BMS}||}$$
\end{thm}

\section{Counting Closed Geodesics}
In this section we prove Theorem 1.2.
Denote by $G_{h}$ the set of pseudo-Anosov elements of $G$.
Denote by $G_{hp}$ the set of primitive pseudo-Anosov elements of $G$.
Let $\Omega(l)$ be the set of closed primitive geodesics on $Teich(S)/G$ of length at most $R$. For $g\in \Omega(l)$ let $D_{g}$ be the Lebesgue measure on $g$ normalized to unit mass.

We will prove:
\begin{thm}
$$\lim_{t\to \infty}\delta t e^{-\delta t}\sum_{g\in\Omega(t)}D_{g}=||\mu||^{-1}\mu$$
\end{thm}
Theorem 1.2 is an immediate corollary.
\begin{proof}[Proof of Theorem 13.1]
Let $x\in Teich(S)$.
Denote by $V(x,r)\subset Teich(S)^{2}\cup PMF^{2}$ the set of pairs $(a,b)$ such that $[a,b]\cap B(x,r)\neq \emptyset$. 
Recall the measure $\widetilde{\mu}$ on $\Lambda(G)\times \Lambda(G)$ by $$d\tilde{\mu}(\eta,\zeta)=\exp(\delta(G)\rho_{x}(\eta,\zeta))d\nu_{x}(\eta)d\nu_{x}(\zeta)$$ and let
$$d\nu^{t}_{x,1}=d\nu^{t}_{x,x}=\delta ||\mu|| e^{-\delta t}\sum_{\gamma \in G, d(x,\gamma y)\leq t}D_{\gamma x}\otimes D_{\gamma^{-1} x}$$  
$$d\nu^{t}_{x,2}=\delta ||\mu|| e^{-\delta t}\sum_{\gamma \in G_{h}, d(x,\gamma y)\leq t}D_{\gamma x}\otimes D_{\gamma^{-1} x}$$
$$d\nu^{t}_{x,3}=\delta ||\mu|| e^{-\delta t} \sum_{\gamma \in G_{h}, d(x,\gamma y)\leq t}D_{\gamma^{+}}\otimes D_{\gamma^{-}}$$
where $\gamma^{\pm}\in PMF$ denote the stable and unstable laminations of $\gamma$. Note that for $D=d(x,WH(G))$ we have $\nu^{t}_{x,i}$ and $\nu_{x}\times \nu_{x}$  all supported on $(N_{D}WH(G)\cup \Lambda(G))^{2}$ and $V(x,r)\cap (N_{D}WH(G)\cup \Lambda(G))^{2}$  is closed in $Teich(S)\cup PMF$.
\begin{lem}
For every $c>0$ there exists a $t_{0}=t_{0}(x,r,c)>0$ such that if $\gamma \in G$ with $d(x,\gamma x)>t_{0}$ and $(\gamma x, \gamma^{-1}x)\in V(x,r)$  we have that $\gamma$ is pseudo-Anosov and $\rho_{x}(\gamma^{\pm 1}x,\gamma^{\pm})>c$.
\end{lem}
\begin{proof}
Suppose $\gamma \in G$ is a pseudo-Anosov, $p$ a point on the axis of $\gamma$ and $l$ the unit speed parametrization of the axis starting at $p$ in direction $\gamma^{+}$.
Note, $$\rho_{x}(\gamma x, \gamma^{+})=\lim_{t\to\infty}d(x,\gamma x)+d(x,l(t))-d(\gamma x,l(t))=$$ $$d(x,\gamma x)+d(x,l(t))-d(x,l(t-l(\gamma)))\geq d(x,\gamma x)$$ and similarly
$$\rho_{x}(\gamma^{-1} x, \gamma^{-})\geq d(x,\gamma x)$$
Note, as $G$ is convex cocompact it contains no parabolic elements. Since $G$ is a hyperbolic group, it contains only finitely many conjugacy classes of finite order elements. Therefore, the fixed points of finite order elements of $G$ are contained in finitely many $G$ orbits of $Teich(S)$. Let $D$ be the maximum distance of these orbits from $WH(G)$. Suppose now that $\gamma_{n}\in G$ is a sequence of finite order elements with $d(x,\gamma_{n}x)\to \infty$ and $(\gamma_{n} x, \gamma^{-1}_{n}x)\in V(x,r)$. Let $p_{n}$ be the fixed point of $\gamma_{n}$. Then $x,\gamma_{n}x,\gamma^{-1}_{n}x$ all lie on the same circle $C_{n}$ of radius $r_{n}\to \infty$ centered at $p_{n}$. 

Taking a subsequence, we can assume $$p_{n}\to \eta \in PMF$$ and $$y_{n}=\gamma_{n}x\to \zeta \in PMF$$ and $$z_{n}=\gamma^{-1}_{n}x\to \theta \in PMF$$

Then clearly $\zeta, \theta \in \Lambda(G)$.

 Also, the $p_{n}$ are all contained in $N_{D}WH(G)$ and thus $\eta \in \Lambda(G)\subset UE$. We claim $\zeta=\eta =\theta$, which would imply that for large enough $n$ we have $(\gamma_{n}x, \gamma^{-1}_{n}x)\notin V(x,r)$, contradicting our assumption.  
Indeed, $$\rho_{x}(\eta,\zeta)=\lim_{n\to\infty}d(x,p_{n})+d(x,y_{n})-d(y_{n},p_{n})=\lim_{n\to\infty} r_{n}\to \infty$$
which is impossibe if $\eta\neq \zeta$ by continuity of $\rho$. Similarly, $\eta=\theta$.
\end{proof}
For the remainder of the argument, the proof of Roblin's Theorem 5.1.1 carries through with essentially no modification. From Lemma 13.2 we obtain
\begin{cor}
When restricted to $V(x,r)$ we have $\nu^{t}_{x,i}-\nu^{t}_{x,j}\to 0$ for $i,j=1,2,3$.
\end{cor}

Note, for $\eta,\zeta \in V(x,r)$ we have $0<\rho_{x}(\eta,\zeta)\leq 2r$ and thus by Theorem 12.8 and Lemma 13.2 for any positive continuous $\psi$ compactly supported on $V(x,r)$ we have 
$$e^{-2\delta r}\int \psi d\widetilde{\mu}\leq \lim \inf \int \psi  d\nu^{t}_{x,3}\leq \lim \sup \int \psi  d\nu^{t}_{x,3}\leq \int \psi d\widetilde{\mu}$$ as $t\to \infty$.
For $\gamma \in G_{h}$ let $g_{\gamma}$ be the axis of $\gamma$, and denote by $L_{\gamma}$ the arclength measure on $\gamma$. Let $l(\gamma)$ denote the translation length of $\gamma$ in the Teichm$\ddot{\mathrm{u}}$ller metric.
Let $$M^{t}_{x}=\delta e^{-\delta t}\sum_{\gamma \in G_{h}, l(\gamma)\leq t} L_{\gamma}$$ and
 $$M^{t}_{x,3}=\delta e^{-\delta t}\sum_{\gamma \in G_{h}, d(x,\gamma x)\leq t} L_{\gamma}$$
Note that $$l(\gamma)\leq d(x,\gamma x)\leq l(\gamma)+2d(x, g_{\gamma})$$ and thus when restricted to $V(x,r)$ we have 
$$M^{t}_{x,3}\leq M^{t}_{x}\leq e^{2\delta r}M^{t+2r}_{x,3}$$
Let $\widehat{V}(x,r)\subset Q^{1}(S)$ denote all quadratic differentials on geodesic segments defined by elements of $V(x,r)$. Note, $$M^{t}_{x,3}=||\mu^{BMS}||^{-1}\nu^{t}_{x,3}\otimes ds$$ and thus for any $\phi \in C^{+}_{c}(\widehat{V}(x,r))$ we have  
$$e^{-2\delta r}||\mu||^{-1}\int \phi d\mu\leq \lim \inf \int \phi dM^{t}_{x,3}\leq \lim \sup \int \phi dM^{t}_{x,3}\leq ||\mu||^{-1}\int \phi d\mu$$ 
Denote by $G_{hp}\subset G$ the set of primitive hyperbolic isometries so that 
$$M^{t}_{x}=\delta e^{-\delta t} \sum_{\gamma \in G_{hp}, l(\gamma)\leq t}\left\lceil \frac{t}{l(\gamma)} \right\rceil L_{\gamma}$$
Clearly $$M^{t}_{x} \leq E^{t}:=\delta t e^{-\delta t}\sum_{\gamma\in G_{hp},l(\gamma)\leq t}\frac{1}{l(\gamma)}L_{\gamma}$$
Moreover, note $\left\lceil \frac{t}{l(\gamma)} \right\rceil \geq \frac{1}{l(\gamma)}$ whenever $1\leq l(\gamma)\leq t$ and $t\geq 2$, so for any $\phi \in C^{+}_{c}(\widehat{V(x,r)})$ we have 
$$\sum_{\gamma\in G_{hp},l(\gamma)\leq t}\frac{1}{l(\gamma)}\int \phi dL_{\gamma}=O(e^{\delta t})$$ since $\int \phi dM^{t}_{x}$ is bounded as $t\to \infty$.

Now, if $e^{-r}t<l(\gamma)\leq t$ then  $$\left\lceil \frac{t}{l(\gamma)} \right\rceil \geq 1\geq \frac{e^{-r}t}{l(\gamma)}$$ and so 
$$\int \phi dM^{t}_{x}\geq e^{-r}\delta t e^{-\delta t}\sum_{\gamma\in G_{hp},e^{-r}t<l(\gamma)\leq t}\frac{1}{l(\gamma)}\int \phi dL_{\gamma}=$$ $$e^{-r}\int \phi dE^{t}-e^{-r}\delta t e^{-\delta t}\sum_{\gamma\in G_{hp},l(\gamma)\leq e^{-r}t}\frac{1}{l(\gamma)}\int \phi dL_{\gamma}$$

By above remarks, the second term in the above difference is bounded above by a constant multiple of $$te^{\delta (e^{-r}-1) t}=o(1)$$
Thus we get 
$$\lim \sup \int \phi dM^{t}_{x}\geq e^{-r}\lim \sup \int \phi dE^{t}$$
Putting everything together we obtain as $t\to \infty$
$$e^{-2\delta r}||\mu||^{-1}\int \phi d\widetilde{\mu}\leq \lim \sup \int \phi dE^{t} \leq \lim \sup \int \phi dE^{t}\leq e^{(2 \delta+1)r}||\mu||^{-1}\int \phi d\widetilde{\mu}$$
Choosing a partition of unity subordinate to the locally finite cover of $Q^{1}(S)$ by the $\widehat{V(x,r)}$ (where $r>0$ is fixed) we obtain the above relation for each $\phi \in C^{+}_{c}(Q^{1}(S))$. Letting $r\to 0$ completes the proof.
\end{proof}

\end{document}